\documentclass[14pt]{extarticle}
\usepackage{extsizes}
\usepackage{amssymb,amsfonts,amsthm}
\usepackage{amsmath}
\usepackage[makeroom]{cancel}
\usepackage{mathtools}
\usepackage{mathrsfs,pifont}
\usepackage{slashed,mathabx} 
\usepackage[bbgreekl]{mathbbol}

\usepackage[top=1in, bottom=1.25in, left=0.75in,
right=0.75in]{geometry}

\usepackage[font={small,sl}]{caption}
\usepackage[all]{xy}

\usepackage{hyperref}
\usepackage[backrefs,msc-links]{amsrefs}

\newcommand{\C}{\mathbb{C}}
\newcommand{\F}{\mathbb{F}}
\newcommand{\A}{\mathbb{A}}
\newcommand{\B}{\mathbb{B}}
\newcommand{\Z}{\mathbb{Z}}
\newcommand{\Q}{\mathbb{Q}}
\newcommand{\R}{\mathbb{R}}
\newcommand{\N}{\mathbb{N}}

\newcommand{\bV}{\mathsf{V}}
\newcommand{\Fock}{\mathsf{Fock}}
\newcommand{\bL}{\mathsf{L}}

\newcommand{\bv}{\mathsf{v}}

\newcommand{\aroof}{\widehat{\mathsf{a}}}

\newcommand{\bT}{\mathsf{T}}
\newcommand{\bZ}{\mathsf{Z}}

\newcommand{\bP}{\mathbb{P}}
\newcommand{\bG}{\mathsf{G}}
\newcommand{\bE}{\mathsf{E}}
\newcommand{\bEN}{\mathbb{E}}
\newcommand{\bS}{\mathsf{S}}

\newcommand{\cK}{\mathscr{K}}
\newcommand{\cL}{\mathscr{L}}

\newcommand{\cO}{\mathscr{O}}
\newcommand{\cE}{\mathscr{E}}
\newcommand{\cU}{\mathscr{U}}

\newcommand{\cM}{\mathscr{M}}
\newcommand{\Mbar}{\overline{\cM}}

\newcommand{\cT}{\mathscr{T}}
\newcommand{\cZ}{\mathscr{Z}}

\newcommand{\be}{\mathbf{e}}
\newcommand{\vbe}{\vec{\be}}

\newcommand{\cI}{\mathscr{I}}

\newcommand{\lan}{\left\langle} 
\newcommand{\ran}{\right\rangle}

\newcommand{\fg}{\mathfrak{g}}

\newcommand{\rdd}{/\!\!/}
\newcommand{\Ct}{\mathbb{C}^\times}

\newcommand{\Hd}{{H}^{\raisebox{0.5mm}{$\scriptscriptstyle \bullet$}}}
\newcommand{\Hhd}{{H}_{\raisebox{0.5mm}{$\scriptscriptstyle \bullet$}}}
\newcommand{\bSd}{{\mathsf{S}}^{\raisebox{0.5mm}{$\scriptscriptstyle
      \bullet$}}}
\newcommand{\bSdh}{{\widehat{\mathsf{S}}}^{\raisebox{0.5mm}{$\scriptscriptstyle
      \bullet$}}}
\newcommand{\Extd}{{\Ext}^{\raisebox{0.5mm}{$\scriptscriptstyle \bullet$}}}

\newcommand{\glh}{\widehat{\mathfrak{gl}}} 
\newcommand{\slh}{\widehat{\mathfrak{sl}}} 
\newcommand{\glhh}{\widehat{\widehat{\mathfrak{gl}}}} 
 
\newcommand{\gh}{\widehat{\mathfrak{g}}}

\newcommand{\KM}{\textup{\tiny\textsf{KM}}}
\newcommand{\MO}{\textup{\tiny\textsf{MO}}}
\newcommand{\ADE}{\textup{\tiny\textsf{ADE}}}

\newcommand{\debar}{\bar\partial} 
\newcommand{\bx}{\textup{\ding{114}}}

\DeclareMathOperator{\Mc}{Mc}

\DeclareMathOperator{\Def}{Def}
\DeclareMathOperator{\Chow}{Chow}

\DeclareMathOperator{\Obs}{Obs}
\DeclareMathOperator{\Lie}{Lie}

\DeclareMathOperator{\Hilb}{Hilb}
\DeclareMathOperator{\PT}{PT}
\DeclareMathOperator{\DT}{DT}
\DeclareMathOperator{\MM}{\mathsf{M2}}
\DeclareMathOperator{\End}{End}
\DeclareMathOperator{\Ext}{Ext}
\DeclareMathOperator{\Hom}{Hom}
\DeclareMathOperator{\Aut}{Aut}
\DeclareMathOperator{\rk}{rk}

\DeclareMathOperator{\supp}{supp}

\DeclareMathOperator{\virdim}{vir\, dim}

\newcommand{\vir}{\textup{vir}}
\newcommand{\vtx}{\textup{vtx}}

\newcommand{\pt}{\textup{pt}}

\newcommand{\cF}{\mathscr{F}}

\newcommand{\cG}{\mathscr{G}}

\newcommand{\tO}{\widehat{\mathscr{O}}}

\DeclareMathOperator{\Coker}{Coker}
\DeclareMathOperator{\tr}{tr}

\DeclareMathOperator{\Euler}{Euler}

\newcommand{\CS}{\mathsf{CS}}
\newcommand{\YM}{\mathsf{YM}}

\newtheorem{Proposition}{Proposition}[section]
\newtheorem{Lemma}[Proposition]{Lemma}

\newtheorem{Theorem}{Theorem} 
\newtheorem{Conjecture}{Conjecture}

\theoremstyle{definition}

\begin{document}

\title{Takagi lectures on Donaldson-Thomas theory} 
\author{Andrei Okounkov}
\date{} 
\maketitle

%\begin{abstract}
%\end{abstract}

\setcounter{tocdepth}{2}
\tableofcontents

%\section*{Foreword} 

\bigskip
\bigskip
\noindent 
My goal in these notes is to explain the following two sides of 
DT counts of curves in
algebraic threefolds: 
\begin{itemize}
\item[---] the counts are defined in very general
  situations, and this generality gives the subject its flexibility
  and technical power, while also 
\item[---] the counts are something concrete and natural, once the
  general definitions are specialized to important special cases. 
\end{itemize}
I believe that the combination of these features make the subject particularly rich, and
certainly both of them are important for the multitude of connections
that the field has with other branches of mathematics and 
mathematical physics. 

In these lectures, I am aiming to get to what I consider an exciting
recent progress in the field, namely the determination of the 
K-theoretic counts, in Sections \ref{s_toric} and \ref{s_actual}. 
Given the amount of background material needed, we are not going to
get much time for admiring the view from the top of that hill. 
I hope an interested reader will open more specialized notes
\cites{PCMI,slc}.

\section{What is the DT theory ?} 

%
%Perhaps one should start by  properly explaining the word 
%\emph{theory}. 
Normally, in mathematics, we call a theory
a set of ideas and a certain body of knowledge united by the commonality of
applications, tools, and creators.  Galois theory, for example, helps
one solve some algebraic equations and cope with one's inability to
solve the remaining ones. While there is, certainly, a rapidly growing
body of knowledge in DT theory, and an equally rapidly 
expanding scope of its application, the original
use of the word \emph{theory} here rather followed the physics
tradition\footnote{I believe the term \textsl{DT theory} was used for
  the first time in \cite{MNOP1}.}. 

\subsection{What is a theory ? }\label{s_what_theory}

To a theoretical physicist, a theory is a procedure that can somehow,
perhaps approximately, compute a measurably quantity from the cast of
characters of a physical process and the geometry of the space-time
where this process is taking place. This typically 
involves some infinite-dimensional integration and a \emph{theory} could 
simply mean a specification of the integrand, ideally together with a
prescription for making sense of the integral. Seen from the math
viewpoint, a physical theory is a set of questions, and not necessarily a
set of answers. 

The integral above is typically over some space of fields (that is,
sections of suitable vector bundles) on the space-time manifold $M$ subject to 
certain boundary conditions on $\partial M$. For a mathematician, it
is perhaps easiest to relate to the statistical field
theory, where we can take finer and finer 
combinatorial approximations to $M$, and even similarly discretize the
range of fields if so desired. One is then standing on the solid
ground of finite-dimensional or even plainly finite probability
theory\footnote{The emergence of continuous structures in the small
  mesh/large scale 
limit may then be treated as a miracle. Surely, our ancestors sensed a similar
miracle as they were figuring out the relation between an integral and its 
Riemann sums.}. 

\begin{figure}[!htbp]
  \centering
   \includegraphics[scale=0.75]{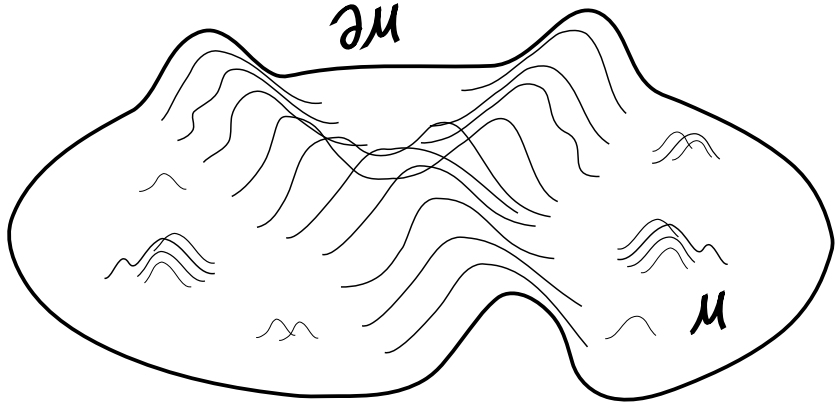}
 \caption{A physical theory has a spacetime manifold $M$ and fluctuating degrees of freedom
   that are subject to boundary conditions on $\partial M$. }
\label{f_fluc}
\end{figure}

A different scenario for cutting the integration down to finite
dimensions comes
about in \emph{supersymmetric} theories. These have a fermionic
symmetry that preserves the functional integral and makes the 
contributions of all nonsupersymmetric field configurations cancel
out, at least formally. The remaining supersymmetric 
configurations form, typically, a
countable union of finite-dimensional manifolds, over each one of
which one can in principle 
integrate. Summing up all these contributions is not
unlike taking the small mesh limit in the 
probabilistic setting\footnote{
And just like in the probabilistic setting, the extra properties that
such sums possess strengthen one's belief in the mathematical reality of
physical theories.}. 

\begin{figure}[!htbp]
  \centering
   \includegraphics[scale=0.5]{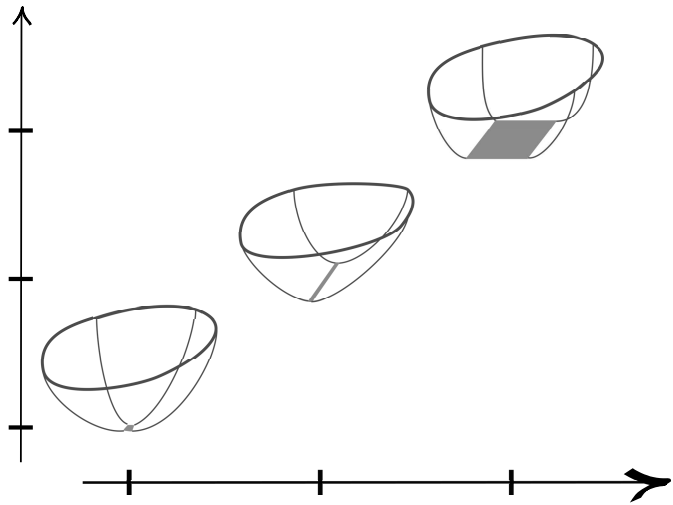}
 \caption{Energy landscape in supersymmetric theories: susy configurations minimize the energy for given
   topology, e.g.\ complex curves have minimal area,  instantons have
   minimal Yang-Mills action etc. For fixed topology, they form finite-dimensional moduli
   spaces, of growing dimensions.}
\label{f_super}
\end{figure}

\subsection{Boundaries and gluing}\label{s_boundaries} 

Whatever the exact flavor, a physical theory wants to have: 
\begin{itemize}
\item[---] a space-time manifold $M$, smooth or combinatorial, often
  with a metric, or some other fixed structures, and 
\item[---] some degrees of freedom that are constrained on the boundary $\partial M$ and fluctuate in
the bulk of $M$.  
\end{itemize}
The integral over these fluctuations defines a 
map\footnote{often called \emph{partition functions}
as a matter of habit, really. This custom has a certain charm in our context,
since there is indeed a lot of partitions around.}
$$
\bZ(M, \partial M) :\,\,  \textup{boundary conditions } 
\xrightarrow{} 
\textup{numbers} 
$$
which in a theory with  \emph{local} interactions has to satisfy the following basic gluing property. 

Let $S\subset M$ a hypersurface that avoids $\partial M$ and let $M_1$
and $M_2$ be the manifolds with boundary into which $S$ cuts $M$. In
the functional integral, we can first fix the degrees of
freedom along $S$ and then integrate over them. In a local theory, the
fluctuations in $M_1$ and $M_2$ will be independent once they are fixed
along $S$, giving
\begin{equation}
\bZ(M, \partial M) = \int \bZ(M_1, \partial_1 M \sqcup S) \, 
 \bZ(M_2, \partial_2 M \sqcup S) \,,\label{gluing}
 \end{equation}
where $\partial_i M = \partial M \cap M_i$ and the integral is over all possible boundary conditions on
$S$. To be sure, the integral in \eqref{gluing} is at this point purely
symbolical and is technically best expressed as a pairing between
suitable functional spaces associated with $S$. 

\begin{figure}[!htbp]
  \centering
   \includegraphics[scale=0.75]{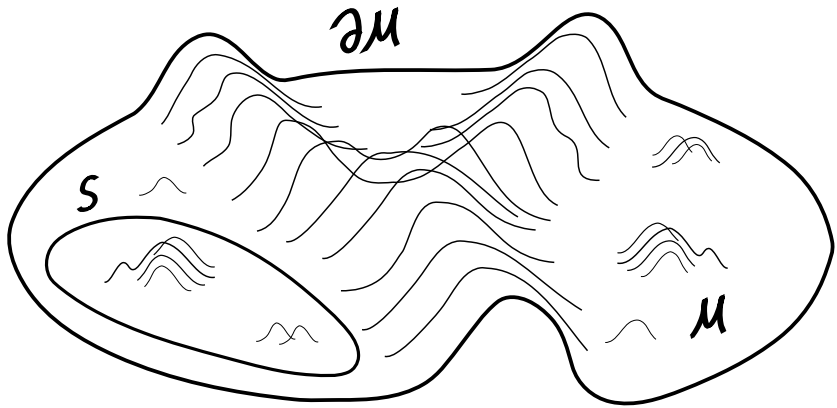}
 \caption{In a theory with local interactions, fluctuations on two
   sides of a hypersurface $S$ are independent once boundary conditions are imposed
 along $S$.}
\label{f_fluc_glue}
\end{figure}

The equality \eqref{gluing} is very powerful as it: 
\begin{itemize}
\item[---] lets one compute $\bZ(M,\partial M)$ by cutting $M$
  into small pieces, and 
\item[---] constraints these pieces thought the independence of 
\eqref{gluing} on the choice of $S$. \end{itemize}
It gains real strength if $\bZ(M,\partial M)$ factors through
some large equivalence relation like e.g.\ orientation preserving
diffeomorphisms, as this really reduces the number of standard
pieces. This is the case for \emph{topological} quantum field
theories \cites{AtTQFT, Kap, WitTQFT}, the Chern-Simons theory of real 3-folds being a very
important example. 

\subsection{Boundaries and gluing in algebraic geometry} 

\subsubsection{}

We want something like the setup of Section \ref{s_boundaries}
but for a smooth quasiprojective algebraic
variety $X$ in place of a smooth real manifold $M$. We don't loose or
gain anything by taking coefficients in $\C$, and for complex $X$ the
story can, in principle, be squeezed into the above real
mold. However, I think the cognitive efficiency is maximized here at the
level of analogies, not precise matches. 

We will allow 
noncompact $X$ as longs as objects that we want to count (basically,
complete cuves in $X$) form moduli spaces that are either
proper or proper over some other space of interest. Further, 
\emph{equivariant} counts with respect to some torus 
$$
\bT \subset \Aut(X)
$$
may be defined as long as $\bT$-invariant curves have compact moduli. 

\subsubsection{}

In place of $\partial M$, we will take an arbitrary smooth divisor
$D\subset X$. This is a very inclusive interpretation of an algebraic
analog of $\partial M$, people often insist on calling an effective 
anticanonical divisor the boundary of $X$, 
see in particular \cites{KhRos}. While this is certainly a very important special
case,  
there is
no reason for us not to look beyond it. 

\subsubsection{}

The analog of topological invariance will be the invariance of
$\bZ(X,D)$ with respect to deformations of the pair $(X,D)$.
This deformation will have to be equivariant for 
$\bT$-equivariant counts. 

\subsubsection{}

The analog of \eqref{gluing} will hold when $X$ degenerates into a
union of $X_1$ and $X_2$ as follows 
\begin{equation}
\xymatrix{ 
X \ar@{^{(}->}[r] \ar[d]& \widetilde{X} \ar[d] &  X_1 \cup_D X_2 \ar[d]
 \ar@{_{(}->}[l] \\
1 \ar@{^{(}->}[r] & \A^1&  0 \,, 
 \ar@{_{(}->}[l]
}\label{degen}
\end{equation}
where $\A^1$ is the base of the deformation, the total space
$\widetilde{X}$ is smooth, and the special fiber is the transverse
union of two components along a smooth divisor $D$.  A one-dimensional 
example of this is 
$$
\{ xy = 1 \}  \subset \{ xy = t \}  \supset \{xy=0\}\,, 
$$
where $t$ gives the map to $\A^1$, and this is what it looks like in
general in directions transverse to $D$. In \eqref{degen}, we assume
that 
$D$ is disjoint from any other divisors chosen in $X$, in parallel to
the $S \cap \partial M = \varnothing$ assumption in Section 
\ref{s_boundaries}. 

\begin{figure}[!htbp]
  \centering
   \includegraphics[scale=1]{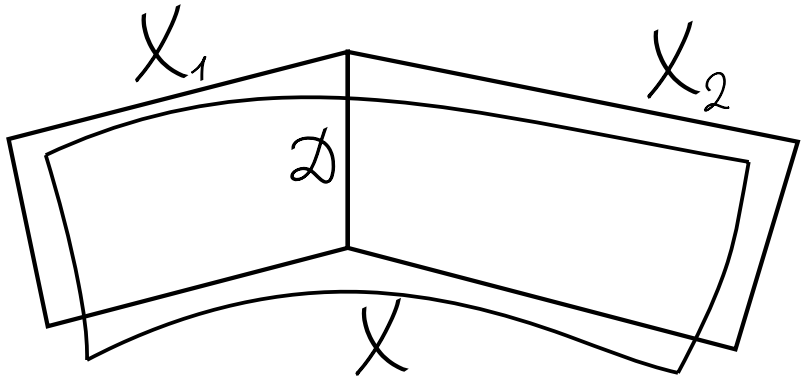}
 \caption{The algebraic analog of cutting a manifold along a
   hypersurface is a degeneration of $X$ to a transverse union of
$X_1$ and $X_2$ along a divisor.}
\label{f_degen}
\end{figure}

\subsubsection{}

An important example of such degeneration is 
\begin{equation}
\xymatrix{ 
X \ar@{^{(}->}[r] \ar[d]& 
\textup{Blowup}_{D\times \{0\}} (X \times \A^1) \ar[d] &  X \cup_D 
\B \ar[d]
 \ar@{_{(}->}[l] \\
1 \ar@{^{(}->}[r] & \A^1&  \,\,\, 0 \,, 
 \ar@{_{(}->}[l]
}\label{blowoff}
\end{equation}
which can be considered for any smooth divisor $D \subset X$. 
Here 
\begin{equation}
  \label{defB}
  \B = \bP(N_{X/D} \oplus \cO_D)
\end{equation}
is $\bP^1$-bundle over $D$ associated to the 
rank 2 vector bundle $N_{X/D} \oplus \cO_D$, where
$N_{X/D}$ is the normal bundle to $D$ in $X$. 

The manifold \eqref{defB} is an algebraic analog of 
$$
\begin{matrix}
\textup{tubular neighborhood}\\
\textup{of $\partial M$ in $M$} 
\end{matrix} \quad \cong \quad \partial M \times [0,1] 
$$
and similarly to this it has two boundary pieces isomorphic to $D$
namely 
\begin{equation}
D_0 = \bP(N_{X/D}) \subset \B \supset \bP(\cO) = D_\infty
\,. \label{D0Dinf}
\end{equation}
Practitioners of DT theory like to give affectionate nicknames to the
objects of the study, but there is little or no consistency in this.
We will call $\B$ \emph{bubble} in these notes, this is not a name that is
in common use. 

\begin{figure}[!htbp]
  \centering
   \includegraphics[scale=0.64]{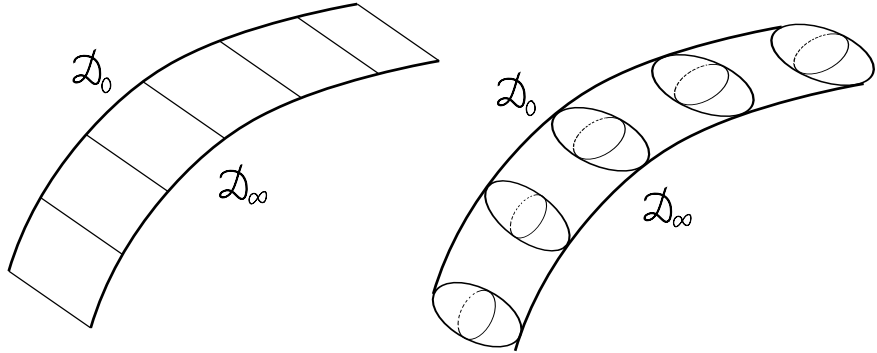}
 \caption{An algebraic geometer's and a complex geometer's impression
   of the bubble $\B$. This is the algebraic analog of a tubular
   neighborhood of $\partial M$ in $M$.}
\label{f_bubble}
\end{figure}

\subsubsection{}

There is a natural
equivalence relation induced by \eqref{degen}, namely 
\begin{equation}
[X] = [X_1] + [X_2] - [ \B] \label{LP}
\end{equation}
where
$$
\B = \bP(N_{X_1/D} \oplus \cO_D) = \bP(N_{X_2/D} \oplus \cO_D)  \,. 
$$
The two formulas for $\B$ are equivalent because 
$N_{X_1/D} \cong N_{X_2/D}^{\vee}$.

A very powerful result of Levine and Pandharipande \cites{LP} is that the 
equivalence relation \eqref{LP} 
generates all relations of the algebraic cobordism, and, in particular, any projective $X$ may be linked to a product of projective spaces by a
sequence of such degenerations. While \eqref{LP} 
by itself does \emph{not} reduce the
DT theory to that of 
\begin{equation}
X \in \left\{\bP^3,\bP^2 \times \bP^1,\left(\bP^1\right)^3\right\}
\,,\label{prodP}
\end{equation}
it limits the number of standard pieces from which all
other geometries may be assembled. 

\subsubsection{}

In the presence of a group action, e.g.\ for the toric varieties in 
\eqref{prodP}, equivariant localization gives a
way to cut $X$ further into really noncompact pieces, and this will be
very 
important for us in Section \ref{s_toric}. 

\subsection{What fluctuates ?}

\subsubsection{} 

Donaldson-Thomas theory is for algebraic threefolds
$$
\dim X = 3 
$$
and in the original vision of \cites{DonTh} the fluctuating degrees of
freedom are given by a coherent sheaf $\cF$ on $X$, or maybe a complex
thereof. For a second, one can imagine that $\cF$ is an algebraic
vector bundle on $X$, although the main role in what follows will be
played by sheaves in a sense diametrically opposite to vector
bundles. 

The inspiration for this came, among other things, from the 
Chern-Simons theory of real 3-folds, and also from Donaldson's theory
for real 4-folds. The latter involves integration over moduli spaces of 
instantons, that is, connections with antiselfdual
curvature. For a K\"ahler surfaces $S$ viewed as a real 4-fold, a
theorem of Donaldson \cites{DonKr} identifies instantons with stable
holomorphic vector bundles on $S$, and so one is really integrating
over those. On the other hand, holomorphic bundles on a complex 3-fold
are critical points of the holomorphic Chern-Simons functional \eqref{hCS},
which links DT and CS theories. 

A most superficial review of these
points will be attempted in Section \ref{s_before} below. 

\begin{figure}[!htbp]
  \centering
   \includegraphics[scale=0.46]{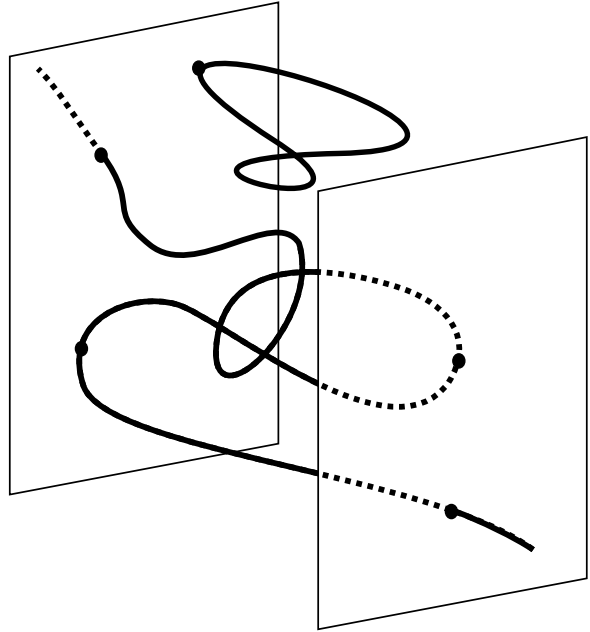}
 \caption{For us, the DT theory is the theory of fluctuating algebraic
 curves in $X$, constrained by how they meet fixed divisors $D\subset X$.}
\label{f_DT}
\end{figure}

\subsubsection{} 

In these notes, however, we won't see any vector bundles on $X$, and
instead we will be looking at sheaves that look like the first or,
equivalently, the last term in an exact sequence of the form 
\begin{equation}
  \xymatrix{
&& \textup{\small functions on $X$}  \ar@{=}[d]\\ \vspace{-3pt}
0 \ar[r] &\cI_C \ar[r] & \cO_X\ar[r] & \cO_C\ar[r] & 0 \\
& \textup{\small equations of $C$}  \ar@{=}[u] && \textup{\small functions on $C$}  \ar@{=}[u]
}  \label{IOO} 
\end{equation}
where $C\subset X$ is a curve, or more precisely, a $1$-dimensional projective
subscheme. Of course, in general a subscheme is precisely defined from
a subsheaf or a quotient of $\cO_X$ as in \eqref{IOO}, but at the first
encounter it may be useful to think of a nice smooth curve which is embedded in
$X$ and then allowed to move. 

The moduli space for complexes \eqref{IOO} is the good old Hilbert
scheme $\Hilb(X,\textup{curves})$ of Grothendieck, see e.g.\ \cites{FDA,Koll}. The Hilbert scheme
 can be defined for $C$
and $X$ of arbitrary dimension, but there is something new and different
in how one integrates over it, and other moduli of sheaves, for $\dim
X=3$. 

\subsubsection{}

The Hilbert scheme splits into components 
$\Hilb(X,[C],\chi)$ according to the degree and 
the arithmetic genus of $C$ 
$$
([C],\chi(\cO_C))  \in H_2(X,\Z) \oplus H_0(X,\Z) \,, 
$$
and a DT partition function is defined as a generating function over 
these discrete invariants. This means that 
\begin{equation}
\bZ(X,D) =  \sum_{(d,n) \in H_2(X,\Z) \oplus H_0(X,\Z)} 
Q^d \, (-z)^n \,\,\, \int_{\Hilb(X,d,n)} \dots \label{ZXD}
\end{equation}
where $z$ and $Q$ are new variables serving to keep track of the
degree and genus, and $Q^d$ is a multiindex. The minus sign in 
$(-z)^n$ in \eqref{ZXD} is introduced to save on signs in many other 
places. 

\subsubsection{}\label{s_DT_higher_rank}

There is a number of good reasons to focus on curves within the universe of
coherent sheaves, among them:
\begin{itemize}
\item[---] the theory can be defined for very general $X$, in
  particular, without any assumptions\footnote{Special things happen
    if $\cK_X \cong \cO_X$. Some levels of complexity collapse then, while new
    possibilities also open. It is, however, awkward to restrict to
    just that case as it is not preserved by degenerations \eqref{degen}.} on the canonical class $\cK_X$, 
\item[---] the corresponding counts have a multitude of deep connections to
  other enumerative theories as well as other fields of
  mathematics such as representation theory, 
\item[---] I expect, and I imagine that this expectation is shared by
  many, that the curve counts are fundamental and other sheaf counting
  problems may be reduced to them in some effective way. 
\end{itemize}
Certainly, the last part is only a vague expectation and useless
unless made precise. As a transition from curves to higher rank
sheaves one can study theories where e.g.\ more than one section of $\cF$ in
\eqref{cOq} are considered, see e.g. \cites{Bakk}.

\subsection{Integration} 

We need to explain what is meant by the integral in
\eqref{ZXD}. Certainly cutting down the functional integral to an
integral over a finite-dimensional space constitutes major progress,
but the truth is that $\Hilb(X,d,n)$ has an unknown number of
irreducible components of unknown dimensions, so it is not clear how
one could integrate anything over its fundamental class. Human
incompetence in basic geometry of Hilbert schemes starts with the 
simplest possible case $\Hilb(\A^3,0,n)$ of points in a linear space. 
The dimension of this scheme is unknown except for small $n$. 

While describing e.g.\ the irreducible components of $\Hilb(\A^3,0,n)$
is an interesting geometric challenge, this is not the direction in
which the development of the DT theory goes. Going back all the way to
Section \ref{s_what_theory}, the finite-dimensional 
integral over 
$$
\cM = \textup{moduli of susy configurations} 
$$
is really a residue of a
certain infinite-dimensional integral. The space $\cM$ is cut out 
by 
certain equation, typically nonlinear PDEs like the equation
$\debar^2=0$ for a connection on a holomorphic vector bundle. When 
these equations are transverse, $\cM$ is a finite-dimensional
manifold and we may expect the functional integral to
localize to the Lebesgue measure on $\cM$.  

However, in real life, equations are seldom transverse, which
is how one is getting very singular varieties of unknown
dimension whose 
fundamental cycle has nothing to do with the original integral. The 
original integral localizes to a certain \emph{virtual fundamental
  cycle}
\begin{equation}
  \label{vircyl}
  \left[ \cM \right]_\vir \in H_{2 \virdim}
\left(\cM\right)\,, 
\end{equation}
where the virtual, or expected, dimension may be computed from
counting the degrees of freedom and is typically some Riemann-Roch
computation. For example 
$$
\virdim \Hilb(X,d,n)= d \cdot c_1(X) \,. 
$$
Analytically, the virtual cycle may be constructed by a very small
perturbation of the equations that puts them into general position. 
For what we have in mind, this is both too 
inexplicit and will certainly break the symmetries of the
problem that are so important for equivariant counts. 

It is much
better to treat the problem as an excess intersection problem in 
algebraic geometry and construct the cycle using certain vector 
bundles on $\cM$ that describe the deformations and equations at
$\cM$. This data is called an \emph{obstruction theory} and the
construction of virtual cycles from it was given in the fundamental
paper of Behrend and Fantechi \cites{BF}. The obstruction theory for DT
problems was constructed by R.~Thomas in \cites{ThCass}, this is the 
technical starting point of DT theory. We will get a feeling how this
works below. 

With this, \eqref{ZXD} should be made more precise as follows 
\begin{equation}
\bZ(X,D) = \sum_{(d,n) \in H_2(X,\Z) \oplus H_0(X,\Z)} 
Q^d \, (-z)^n \,\,\, \int_{\left[\Hilb(X,d,n)\right]_\vir} \dots \label{ZXD2}
\end{equation}
where dots stand for cohomology classes that come from the boundary 
conditions, as will be discussed in a moment. For now we note that
beyond the virtual fundamental cycle, the DT moduli spaces $\cM$ have
a virtual $\hat A$-genus 
$$
\tO_\vir \in K(\cM) 
$$
with which one can define a K-theoretic analog of \eqref{ZXD2} 
\begin{equation}
\bZ(X,D)_\textup{K-theory} = \sum_{(d,n) \in H_2(X,\Z) \oplus H_0(X,\Z)} 
Q^d \, (-z)^n \,\,\, \chi\left(\tO_{\Hilb(X,d,n),\vir} \otimes \dots \right)\label{ZXDK}
\end{equation}
defined in the $\Aut(X)$-equivariant K-theory.

\subsection{PT theory} 

\subsubsection{}

Abstract DT theory may be described as counting stable objects in
categories that are akin (that is, have similar properties of 
$\Extd$-groups) to coherent sheaves on a smooth threefold. 
These notes are not the place to go into a discussion of 
stability and its variations, which is a deep and technical notion,
see e.g.\ \cites{Br1,Br2,BrAMS,KS}. But one example of how the change of stability can
drastically simplify moduli spaces must be discussed. This is the
moduli space of stable pairs, a simpler cousin of the Hilbert schemes
of curves, first used in the enumerative context by Pandharipande and
Thomas \cites{PT1,PT2}. 

\subsubsection{}

Recall that a point in the Hilbert scheme of curves corresponds to a
coherent sheaf $\cF=\cO_C$ that is a 
$1$-dimensional quotient 
\begin{equation}
  \label{cOq}
  \cO_X \xrightarrow{\,\, s\,\,} \cF \to 0 
\end{equation}
of $\cO_X$. In this arrangement, the map $s$ is very good 
(surjection)
but the sheaf $\cF$ can be quite bad. The problems with $\cF$ start
already for the Hilbert schemes of points, that is, for 
$\dim \cF = 0$, and pollute all curve counts of positive degree.
One can't help feeling that there should a way to disentangle the
contributions of points and curves and, to a large extent, this is
exactly what PT theory achieves. 

In PT theory, one ask less of $s$ and more of $\cF$. What we want from
$\cF$ is to be a \emph{pure} 1-dimensional sheaf, that is, to have no
0-dimensional subsheaves. For instance, if $\dim \supp \cF=0$
then we must have $\cF=0$. What we are willing to allow of $s$ is that 
$$
\dim \Coker s = 0
$$
instead of $\Coker s = 0$. 

For example, if the support of $\cF$ is a reduced smooth curve $C$
then the only stable pairs are of the form 
\begin{equation}
  \label{XsF}
  \xymatrix{
\cO_X \ar[rr]^{\!\! \!\!s\,\,\quad} \ar[dr] && \cF= \cO_C(\textstyle{\sum} p_i)  \\
& \cO_C \ar[ur]} 
\end{equation}
where $\{p_i\} \subset C$ is an unordered collection of points and the
other two maps in \eqref{XsF} are the natural ones. In other words,
the fiber of the PT moduli spaces over the point
$$
[C] \in \Chow(X)=\{\textup{effective $1$-cycles in $X$}\} 
$$
corresponding to a
smooth curve $C\subset X$ 
is $\bSd(C)= \bigsqcup_n \bS^n C$. This is infinitely
simpler than the fiber of the Hilbert schemes of curves. 

\subsubsection{}
Obviously, PT counts cannot agree with the Hilbert scheme counts
because they do not agree already in degree 0. It is, however,
conjectured that they agree once one takes this into account and
divides out the contribution of points from the generating function,
that is 
\begin{equation}
\bZ_\textup{PT} (Q,z) = \bZ_\textup{Hilb} (Q,z)
\bigg/ \bZ_\textup{Hilb}  (0,z) \,. 
\label{DTPT}
\end{equation}
The denominator in \eqref{DTPT} was computed explicitly for an
arbitrary 3-fold for the needs of the GW/DT correspondence. 
This will be discussed in Section \ref{s_deg0}. 

For toric 3-folds, the equality \eqref{DTPT} comes out of the whole
machine that computes their DT counts \cites{MOOP}.  For Calabi-Yau 3-folds,
Y.~Toda gave a proof \cites{Toda1,Toda2} that analyses 
\emph{wall-crossing} connecting 
the Hilbert scheme stability condition to the stability condition for
stable pairs. Wall-crossing techniques have not been really explored outside the
world of Calabi-Yau varieties.

\subsection{Actual counts} 

Actual DT counts are very complex. While they are defined for a
general smooth divisor $D$ in a general smooth threefold $X$, it is
not even clear if there is a universal language in which the answers
can be stated, starting with the case $D=\varnothing$. 

We understand the answers explicitly for very special $X$, like the toric
varieties with $D=\varnothing$ or certain fibrations with 
$
D=\bigcup \textup{\sl Fibers}
$
 that will be discussed in Sections \ref{s_toric} and
\ref{s_actual}. Already these are very rich and seem to require, more
or less, all of the mathematics that falls into the author's area of expertise
for their understanding. 

I view this is a great advantage of the DT theory: there are certainly
whole galaxies of open problems in it of all possible flavors --- from
foundational to combinatorial. One of my goals in these notes is to
help potential explorers of these galaxies to get a sense of what
awaits them there.

\section{Boundary conditions and gluing} 
\label{s_boundary} 

Now we put the boundary conditions in \eqref{ZXD2}, that is,
we constrain the fluctuating curve $C$ by how it intersects our fixed
divisor $D$.

\subsection{Intersection with the boundary}

\subsubsection{} 

Algebraically, functions on $C \cap D$ are the functions on 
$C$ taken modulo the equation of $D$, that is, they are the cokernel in 
\begin{equation}
  \label{Ccap D}
  \cO_C \xrightarrow{\textup{$\times$ equation of D }} \cO_C \to 
\cO_{C\cap D} \to 0 \,. 
\end{equation}
There is an open locus 
\begin{equation}
\Hilb(X, d, n)^{\perp} \subset \Hilb(X, d, n)\label{Hilbo}
\end{equation}
where the first arrow in \eqref{Ccap D} is injective. This is the
correct transversality condition on $C \cap D$ for us\footnote{
The ordinary transversality means that, additionally, $C \cap D$ is
reduced.}. For transverse curves \eqref{Hilbo}, we
have a well-defined map 
\begin{equation}
\cdot \, \cap D : \quad \cO_C \mapsto \cO_{C \cap D} \in 
\Hilb(D, d \cdot [D]) \,, \label{CcapD}
\end{equation}
to the Hilbert scheme of $d \cdot [D]$ points in $D$. 

\subsubsection{}

Contrary to the 3-fold case, 
Hilbert schemes of points in 
surfaces are exceptionally nice algebraic
varieties. In particular, they are smooth and connected of dimension 
$$
\dim \Hilb(D, k) = 2 k \,. 
$$
An introduction to their geometry may be found in \cites{Lehn,
NakL}. As lots of
things in DT theory are built on the geometry of $\Hilb(D,k)$, we will be
often coming back to it in these notes. 

\subsubsection{} 

Imposing boundary conditions means 
integrating along the fibers of \eqref{CcapD} or, equivalently,
pulling back cohomology classes from $\Hilb(D)$ via \eqref{CcapD}. 
Neither is well-defined because 
\eqref{Hilbo} is open, meaning some further details need to be filled in
here. 

There are, in fact, at least 3 ways to do it, as will be
discussed presently. The conventions on what to call these flavors of boundary condition
vary. We will call them \emph{nonsingular}, \emph{relative}, and 
\emph{descendent}, respectively. While they all express the same 
general geometric idea of $C$ meeting $D$ in a particular way, an 
actual geometric translation between them is not trivial and plays 
an important role in the development of the theory, see below. 

\subsubsection{}\label{s_H_Fock} 
Once the technical details are filled in, the special structures in
the (co)homology of the Hilbert scheme serve to organize the DT data
in a very nice way. In particular, one can interpret \eqref{ZXD2} 
as 
\begin{equation}
\bZ(X,D) = (\cdot \cap D)_* \left(\sum
Q^d \, (-z)^n \,\left[\Hilb(X,d,n)\right]_\vir \right)
\in \Hhd(\Hilb(D))[[Q,z]]\,, \label{ZXD3}
\end{equation}
where
\begin{equation}
\Hhd(\Hilb(D)) = \bigoplus_k \Hhd(\Hilb(D),k) \cong 
\Fock \left(\Hhd(D)\right)\label{Fock} \,. 
\end{equation}
The important and powerful identification of (co)homology with 
Fock space\footnote{The definition of a Fock space is recalled in
  \eqref{bcSym} below.}
modelled on the (co)homology of the surface $D$ itself is a classical
result of \cites{NakHart}.

\subsection{Different flavors of boundary conditions}\label{s_flavors}

\subsubsection{Nonsingular bc}

The $\bT$-fixed locus 
$$
\left(\Hilb(X, d, n)^{\perp}\right)^\bT\subset \Hilb(X, d, n)^\perp 
$$
may be compact for some torus $\bT \subset \Aut(X,D)$, in which case
the boundary conditions may be imposed for equivariant counts in
localized equivariant cohomology without further geometric
constructions. 

For example, the bubble \eqref{defB} has a fiberwise $\Ct$ action that 
fixes both divisors in \eqref{D0Dinf} and can be used to impose the
nonsingular boundary conditions at either of them. As an 
  exercise, one should check that imposing nonsingular conditions at
  both $D_0$ and $D_\infty$ gives trivial DT counts, in the sense that 
$$
\left(\Hilb(\B, d, n)^{\perp D_0 \cup D_\infty}\right)^{\Ct} 
\cong 
\begin{cases}
\Hilb(D,k)\,,  & (d,n)= k [ \textup{fiber $\B\to D$} ]\,, \\
\varnothing\,, & \textup{otherwise} \,,
\end{cases}
$$
and that in the former case the virtual cycle pushes forward to 
$$
\textup{diagonal}\subset
 \Hilb(D_0,k) \times \Hilb(D_\infty,k) \,. 
$$

\subsubsection{Relative bc} 

There is a resolution of the map \eqref{CcapD}, that is, a 
diagram of the form 
\begin{equation}
  \label{resol}
  \xymatrix{
& \Hilb(X/D,d,n)\ar[dr]^{\quad \cap D }\\
\Hilb(X,d,n)^\perp\ar[rr]\ar@{^{(}->}[ru]&& \Hilb(D,d\cdot[D])  
}
\end{equation}
in which the new map to $\Hilb(D,d\cdot[D])$ is proper. Here
$\Hilb(X/D,d,n)$ is the Hilbert scheme of curves \emph{relative} the divisor
$D$. It is constructed using J.~Li's theory of expanded degenerations
\cites{Li1}. It 
allows $X$ to sprout off new components as in \eqref{blowoff} as many 
times as needed to keep $C$ transverse to $D$. A more complicated
instance of the same general phenomenon is illustrated in Figure 
\ref{f_curve_degen}. 

Note that in the central fiber of \eqref{blowoff} 
the proper transform of the divisor $D$
escapes to the new component and becomes the divisor $D_\infty$ in 
\eqref{D0Dinf}. This is the new divisor $D$ with which we intersect the subscheme
$C$.

With relative boundary conditions, we integrate over $\Hilb(X/D,d,n)$ in
\eqref{ZXD2} and \eqref{ZXD3}. 

\subsubsection{Descendent bc} 

Instead of trying to pull back cohomology classes from 
$\Hilb(D,d\cdot[D])$, we can construct the corresponding classes
directly on $\Hilb(X,d,n)$ as follows. It is well known that the 
$\Hd(\Hilb(D))$ is generated by the K\"unneth components of the 
Chern classes of $\cO_\cZ$ where 
$$
\cZ \subset \Hilb(D) \times D
$$
is the universal subscheme. We have a well-defined class 
\begin{equation}
\cO_C \otimes_\textup{K-theory} \cO_D \in K^\circ(\Hilb(X) \times
D)\label{CDK}
\end{equation}
where $K^\circ$ is the K-theory of locally free sheaves. It has Chern
classes, the K\"unneth components of which can be inserted in the 
integral over $\Hilb(X)$.

Note that these descendent cohomology classed \emph{do not} factor 
through the relations in $\Hd(\Hilb(D))$. In other words, they are
really not pulled back via some map like \eqref{CcapD}. As we will see
in a minute, they can be translated into the relative conditions, but
the coefficients in this translation are functions of $Q$ and $z$. 

Also note that the structure sheaf of the universal curve 
$$
\cO_C \in K^\circ( \Hilb(X,d,n) \times X) 
$$
also has Chern classes which we can similarly use to produce 
cohomology classes on $\Hilb(X,d,n)$. These can be also inserted into 
the integral in \eqref{ZXD2}. Such insertions constrain the behaviour
of $C$ in the bulk of $M$. 

Note that the supply of descendent classes is, a priori, indexed not by the Fock
space in \eqref{Fock} but by 
$$
\Hd_{GL(\infty)}(\pt) \otimes \Hd(D) \, \rightarrow  \,
\textup{descendent bc} \,,
$$
that is, by arbitrary Chern classes colored by the cohomology of
$D$. Chern classes colored by the cohomology of $X$ produce
descendent bulk insertions.

\subsection{Gluing}\label{s_gluingDT} 

\subsubsection{}

The following geometric considerations lead to the gluing formula in
DT theory. Let $X$ break up in two pieces as in \eqref{degen} and 
in Figure \ref{f_degen}. This is as nice a degeneration of an algebraic
variety as you will ever see and we want to have the nicest
possible degeneration of the corresponding Hilbert schemes. The Hilbert scheme
of the singular variety
$$
X_ 0 = X_1 \cup_D X_2 
$$
while perfectly well-defined, is not nice, because if $C$ meets the 
$D=(X_0)_\textup{sing}$ in a nontransverse way then this causes all
sort of problem including a failure in the obstruction theory. 

The solution, provided by J.~Li's theory of expanded degenerations
\cites{Li1} is to
allow $X_0$ to open up accordions of the form 
$$
X_0[k] = X_1 \cup_{D_1} \underbrace{\B \cup_{D_2} \dots \cup_{D_k} \B}_\textup{$k$ times}  \cup_{D_{k+1}} X_2 
$$
where the bubble $\B=\bP(N_{X_1/D} \oplus \cO_D)$ is already familiar
from \eqref{defB} and the discussion of relative boundary conditions. The divisors 
$$
D_1 \cong D_2 \cong \dots \cong D_{k+1} \cong D 
$$
are the copies of $D$ that appear together with the bubbles $\B$. The 
subschemes $C\subset X_0[k]$ what we allow have the form 
$$
C= C_0 \cup C_1 \cup \dots \cup C_k \cup C_{k+1}
$$
where the components $\{C_i\}$ are transverse to the divisors $\{D_i\}$
and glue along them in the sense that 
$$
C_i \cap D_{i+1} = D_i \cap C_{i+1}\,. 
$$
A new bubble blows up any time transversality is in danger and 
the moduli spaces for different $k$ fit together
into one orbifold 
$$
\Hilb(X_0[\cdot]) = \bigcup_{k\ge 0} \Hilb(X_0[k])_\textup{semistable} \, \Big/ 
(\Ct)^k 
$$
where $\Ct$ acts on the $\bP^1$-bundle fiberwise preserving the two
divisors in \eqref{D0Dinf}. This is the right object to put as a
central fiber in the family of degenerating Hilbert schemes.

\begin{figure}[!htbp]
  \centering
   \includegraphics[scale=1]{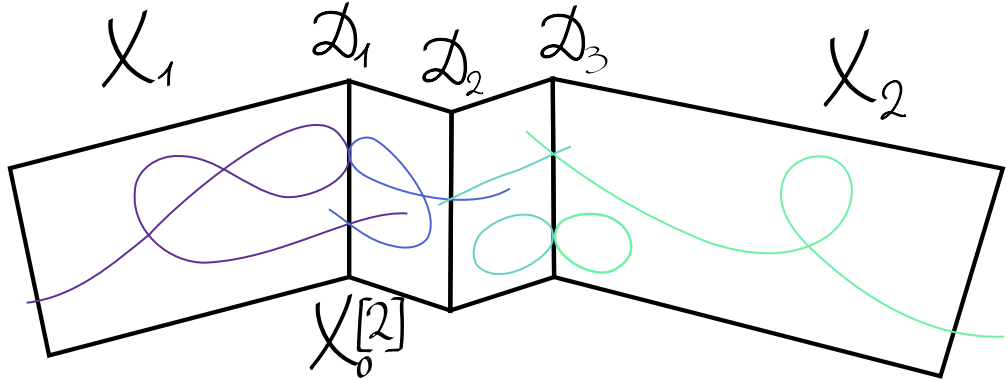}
 \caption{
A curve in an expanded degeneration $X_0[n]$ is a collection of 
subschemes in each
component that are transverse to the singular divisors and glue along
them.}
\label{f_curve_degen}
\end{figure}

\subsubsection{}
The gluing formula for DT proven in \cites{LiWu} is the combination of two
statements. First 
\begin{equation}
\int_{[\Hilb(X)]_\vir} \dots = \int_{[ \Hilb(X_0[\cdot]) ]_\vir} \dots
\,,\label{deform0}
\end{equation}
where the integrand is defined e.g.\ using boundary conditions away
from $D$ or in any other way that makes sense as a cohomology class on
the whole family of the Hilbert schemes. In other words, with the 
right definition, the deformation invariance
of DT counts in the family 
\eqref{degen} includes the central fiber. 

\subsubsection{} 

The second part of the gluing formula says that the integral over the
right-hand side in \eqref{deform0} can be computed in terms of DT
counts in $X_1$ and $X_2$ relative the gluing divisor $D$. Concretely 
\begin{equation}
  \label{Zglue}
  Z(X) = \Big(Z(X_1/D), (-z)^{|\,\cdot\, |} \, Z(X_2/D) \Big)_{\Hhd(\Hilb(D))}
\end{equation}
where we interpret $Z(X_1/D)$ as in \eqref{ZXD} and $|\,\cdot\, |$ is
the grading operator on the Fock space \eqref{Fock}, it acts by $k$ on
the $k$th term in the direct sum in \eqref{Fock}. We write 
$Z(X_1/D)$ to stress the fact that \emph{relative} conditions are
imposed at the divisor $D$. 

The extra weight $(-z)^{|\,\cdot\, |}$ in \eqref{Zglue} is there because 
\begin{equation}
\chi(\cO_C) = \chi(\cO_{C_1}) + \chi(\cO_{C_2}) - 
\chi(\cO_{C_1 \cap D}) \label{chiglDT}
\end{equation}
if $C=C_1 \cup_{C_1 \cap D} C_2$ is a transverse union along the
common intersection with $D$. 

\subsubsection{}
For K-theoretic counts, the first part \eqref{deform0} holds
verbatim, but there is a correction to \eqref{Zglue} discovered in
\cites{Giv1}. It can be phrased by saying that the natural bilinear form
used in \eqref{Zglue} is deformed for K-theoretic counts in a way that
depends on $Q$ and $z$. An introductory discussion of this issue may
be found in Section 6.5 of \cites{PCMI}.

\subsection{Correspondence of 
boundary conditions}\label{s_corresp}

\subsubsection{}

Let $D\subset X$ be a smooth divisor and imagine that some mystery
boundary conditions are imposed on $D$ for curve counts in $X$. We can
bubble off $D$ as in \eqref{blowoff} and conclude from the gluing
formula that 
\begin{multline}
  \label{Mystery}
  Z(X, \textup{mystery at $D$}) = \\
\Big(Z(X/D_0), (-z)^{|\,\cdot\, |} \, Z(\B,
\textup{relative at $D_0$}, \textup{mystery at $D_\infty$})
\Big)_{\Hhd(\Hilb(D_0))} \,. 
\end{multline}
We conclude from \eqref{Mystery} that DT counts in $\B$ with two
flavors of the boundary conditions at $D_0$ and $D_\infty$
respectively precisely give the translation between different boundary
conditions. Note this translation depends on $Q$ and $z$. 

\subsubsection{}

This is completely parallel to how one can translate between boundary
conditions by imposing a full set of alternative boundary conditions
on the other boundary of the tubular neighborhood in Figure
\ref{f_fluc}. 

While the correspondence given by \eqref{Mystery} is a useful
general statement, the concrete identification of 
$Z(\B,D_0,D_\infty)$ may be a
very challenging problem.

\subsubsection{}

Recall that $\B$ has a fiberwise $\Ct$-action that fixes $D_0$ and
$D_1$. Equivariant localization with respect to this action gives 
\begin{multline}
  Z(\B,
\textup{$\mathsf{bc}_0$ at $D_0$}, \textup{$\mathsf{bc}_\infty$ at $D_\infty$})
 = \\
\left(Z(\B,
\textup{$\mathsf{bc}_0$ at $D_0$}, \textup{ns at $D_\infty$}) , 
Z(\B,\textup{ns at $D_0$}, \textup{$\mathsf{bc}_\infty$ at $D_\infty$}) 
\right)_\textup{edge} 
\label{edge_bc} 
\end{multline}
where $\mathsf{bc}_0$ and $\mathsf{bc}_\infty$ denotes some arbitrary 
boundary conditions and the edge inner product on 
$\Hhd(\Hilb(D))$ is a new inner product that reflects the nontriviality of
the fibration 
$$
\pi_\B: \B \to D\,.
$$
It comes from the edge terms in localization formulas, see below, and 
is determined by the deformation theory of ``constant'' curves in $\B$, that is, 
ideal sheaves of the form 
$$
\pi^*_\B(\cI) \in \Hilb(\B) \,, \quad \cI \in \Hilb(D)  \,. 
$$
Using \eqref{edge_bc} one can replace arbitrary boundary conditions by
nonsingular ones in equivariant theory.

\section{Before DT} \label{s_before}

\subsection{Gromov-Witten theory}

There is another curve-counting theory in algebraic geometry which can
be defined for a smooth $X$ of arbitrary dimension. It traces its
origins to the topological strings in the same way as DT theory
descends from the study of supersymmetric gauge theories. 

\subsubsection{Moduli spaces} 

In GW theory, one
integrates over the moduli spaces 
$\Mbar_{g,n}(X,d)$ of \emph{stable maps} to $X$, where a point of
$\Mbar_{g,n}(X,d)$ corresponds to the data 
\begin{equation}
(C,p_1,\dots,p_n) \xrightarrow{\,\, f \, \,} X \label{Mgndata}
\end{equation}
in which 
\begin{align*}
C &\quad \textup{is an at worst nodal curve of genus $g$} \,,\\
p_1,\dots,p_n \in C &\quad \textup{are smooth distinct marked points of
  $C$} \,, \\
f & \quad \textup{is a map of degree $\deg f:= 
f_*[C] = d$\,.} 
\end{align*}
Two maps are isomorphic if there exist a triangle 
\begin{equation}
  \label{isom}
  \xymatrix{ 
(C,\{p_i\}) \ar[rd]_f\ar[rr]^\phi && (C',\{p'_i\})\ar[dl]^{f'}\\
& X 
}
\end{equation}
in which $\phi$ is an isomorphism of pointed curves. As a stability
condition for the data \eqref{Mgndata}, one requires the automorphism
group of \eqref{Mgndata} to be finite. 

The moduli spaces $\Mbar_{g,n}(X,d)$ have a canonical perfect
obstruction theory of virtual dimension
\begin{equation}
\virdim = (g-1)(3-\dim X) + c_1(X) \cdot d + n \,,\label{virdimGW}
\end{equation}
and the corresponding virtual fundamental classes in cohomology. In
their 
K-theory, one has virtual structure sheaves but, in general, one does
\emph{not} have a good virtual $\hat A$-genus, in contrast to DT
theories. 

In GW theory one forms generating functions analogous to \eqref{ZXD2}
\begin{equation}
\bZ_{\textup{GW}}(X; Q,u) = \sum_{(d,g) \in H_2(X,\Z) \oplus \Z} 
Q^d \, u^{2g-2} \,\,\, \int_{\left[\Mbar_{g,n}(X,d)\right]_\vir} \dots \label{ZGW}\,,
\end{equation}
the variable $u$ in which corresponds to the string coupling constant
of topological strings. 

An excellent introduction to stable maps and GW theory may be found in
\cites{mirror_book}.

\subsubsection{Relative GW theory}

A stable map is transverse to a divisor $D$ if the set 
$$
f^{-1}(D) = \{b_i\} 
$$
is finite and disjoint from nodes and marked points in
\eqref{Mgndata}.  Counting multiplicity, we have 
$$
f^* D = \sum \mu_i \, [b_i] \,, \quad \mu_i \in \N=\{ 1,2, \dots\} \,, 
$$
and it is convenient to add the points $\{b_i\}$ and the multiplicities
$\{\mu_i\}$ to the data in \eqref{Mgndata} to get a new moduli space
with morphisms
\begin{equation}
\Mbar_{g,n,\mu}(X/D,d) \owns f \mapsto f(b_i) \in D  \label{fb}
\end{equation}
that take a map $f$ to its $i$th point of tangency with $D$. Expanded
degenerations provide a natural compactification of this relative
moduli space.

As boundary conditions we pull back classes in $\Hd(D)$ via
\eqref{fb}. These are additionally colored by the integers $\mu_i$ and
this color may be recorded by a monomial $t^{\mu_i}$. This gives
\begin{align}
  \textup{relative bc}  &= \bSd \left(\Hd(D) \otimes t \Q[[t]]\right)
  \label{bcSym} \\
& = \Fock(\Hd(D)) \,, \notag 
\end{align}
where the symmetric algebra takes into account the sign rules, that
is, is the symmetric algebra in the category of $(\Z/2)$-graded vector
spaces. 

The parallel with the discussion of Section \ref{s_H_Fock} is not
coincidental. 

\subsubsection{Different dimensions}\label{sHodge} 

Since the GW counts are defined for $X$ of arbitrary dimension, it is
natural to ask how the counts are affected by simple changes of
geometry like going from $X$ to $X\times \A^1$.  

Despite the obvious
noncompactness, $\Ct$-equivariant GW counts in $X\times \A^1$ are
well-defined for the natural $\Ct$ action on $\A^1$. Denoting 
by $t$ the weight of this action, the difference $\delta$ in the obstruction theory
between maps to $X\times \A^1$ and maps to $X$ is 
\begin{multline}
\delta\left(\Def - \Obs \right) 
= \Hd(C, f^*(T\A^1)) = \\
=\Hd(C, t \otimes \cO_C) =
 t - t \otimes H^1(\cO_C) 
= t - t \otimes H^0(\omega_C)^\vee \,. 
\label{ObsHodge} 
\end{multline}
The rank $g$ bundle $H^0(\omega_C)$ is known as the  Hodge bundle over the
moduli spaces of curves. Its fiber over a smooth curve $C$ is formed
by the holomorphic differentials on $C$. Note that the rank in 
\eqref{ObsHodge} is $1-g$ in agreement with the virtual dimension
formula \eqref{virdimGW}. 

This difference in obstruction theories means the insertion of 
$$
t^{-1} \textup{Euler}(t \otimes H^0(\omega_C)^\vee) = 
\sum_k (-1)^k t^{g-k-1} c_k(\textup{Hodge}) 
$$
in the integral over $\Mbar_{g,n}(X)$. Integrals with Chern classes of
the Hodge bundle are known as the \emph{Hodge integrals}. 
We see that such insertions are equivalent to extra factors of $\A^1$
in the target of stable maps. In particular, the equivariant GW theory
of a linear space is the theory of Hodge of integrals over the 
the Deligne-Mumford moduli
spaces of stable curves. 

Note that Hodge integrals are suppressed in the
$t\to \infty$ limit. More
generally, the GW theory of $X^\bT$ may be recovered from GW theory of
$X$ if we send all equivariant variables to infinity. 

\subsubsection{Periodicity} 

If we are willing to increase dimension by 2 then there is a
cleaner relation that does not require taking limits in
$t$. The following vanishing of Chern classes 
\begin{equation}
\textup{Euler}\left(t \otimes \left(\textup{Hodge}\oplus \textup{Hodge}^\vee\right)\right) =
  t^{2g} \label{Mum_vanish} 
\end{equation}
was noticed by Mumford and for smooth curves may be explained by the
existence of the natural flat connection on the rank $2g$ bundle 
$$
H^1(C,\C) = H^1(\cO) \oplus H^0(\Omega^1_C) \,,
$$
see \cites{FP} for a modern discussion. It follows that 
\begin{equation}
\bZ_{\textup{GW}}(X \times \A^2_{(t,-t)} ; Q,u) = \bZ_{\textup{GW}}(X;
Q,i t u) \label{period} 
\end{equation}
where $t \in \Lie \Ct$ acts on $\A^2_{(t,-t)}$ by $
\begin{pmatrix}
t \\
&-t 
\end{pmatrix}
$.

\subsubsection{Critical dimension}

It is clear from \eqref{virdimGW} that the case
$$
\dim X = 3 
$$
is very special in GW theory. Virtual dimension grow/diminish with
genus for $\dim X \gtrless 3$ and are independent of the genus
 for
$\dim X =3$. Precisely for threefolds, one can have interesting
generating functions over all genera with the same insertions, in
particular, with the same boundary conditions. 

As discussed above, equivariant 
GW theory of threefolds also contains the GW
theory of targets of smaller dimension. 

\subsubsection{Compare and contrast}

There is a precise and very nontrivial comparison between the GW
theory in its critical dimension and the DT theory. This comparison
has been one of the guiding stars for the development of the DT theory
and will be discussed properly below. 

Since these lectures are about DT theory, naturally, we will be
stressing those aspects of the DT theory in which it compares
favorably to GW theory. These include, for instance, the ease of
``boxcounting'' localization computations, a feature much appreciated by 
practitioners and students alike. More fundamentally, DT theory is
much better suited for going beyond the computations in cohomology as
its partition functions may be computed not only in K-theory but, in
special circumstances, in even more refined enumerative theories. 

But in fairness to GW theory, one should not forget to stress its
great scope and flexibility. Not only is it defined for smooth
algebraic varieties of any dimension, it can also count
pseudoholomophic surfaces in real symplectic manifolds $M$, and those can
even have boundaries ending on Lagrangian subvarieties of
$M$. Already the counts of pseudoholomophic disks in such situations 
form a structure of stunning richness captured by Fukaya categories of
different flavors. These constrain the counts of holomorphic curves
and, one day, may help with the actual computations.

\subsection{Chern-Simons theory} 

\subsubsection{}

In Chern-Simons theory, the space-time is an oriented
 3-manifold $M$ and one
integrates over gauge-equivalence classes of connections $A$ on a 
$G$-bundle, where $G$ is a simple Lie group such as $G=SU(n)$. 
The integrand is
$\exp(2\pi i k \, \CS(A))$ where 
\begin{equation}
\CS(A) = \frac{1}{8 \pi^2} \int_M  \tr \left( A \wedge dA + 
\tfrac{2}{3}  \, A \wedge A \wedge A \right) \,.  \label{CS}
\end{equation}
The critical points of which are \emph{flat} connections: 
\begin{equation}
\nabla \CS = 0 \quad \Leftrightarrow \quad F_A = 0 \,. \label{critCS}
\end{equation}
More precisely, \eqref{CS} picks up an integer, the degree of the map 
$M \to G$ under a general gauge tranformation, whence the
quantization of the level $k\in \Z$, which is a very important feature of
the theory. This does not affect \eqref{critCS}, nor does it affect the fact
that flat connections make the dominant contribution to the integral
in the semiclassical $k\to \infty$ limit. 

We have 
\begin{equation}\label{flatpi_1} 
\left\{ \textup{flat $A$}\right\} \Big/ \textup{gauge} = 
\Hom(\pi_1(M) \to G)/ G \,, 
\end{equation}
where the $G$-action on the right is by conjugation. As a critical
locus, \eqref{flatpi_1} has expected dimension zero, which is also
confirmed by the deformation theory of flat connections. Correctly 
counting points of \eqref{flatpi_1} represents various
generalization of the Casson invariant, see \cites{Taubes} for a classical
treatment. 

\subsubsection{}
Donaldson-Thomas theory originated \cites{DonTh} from the following 
precise analog of
\eqref{CS} for complex Calabi-Yau threefolds $X$. Flat connections are
now replaced by $\debar$-operators 
$$
\debar_A = \debar + A \,,  \quad  
A \in \{\textup{(0,1)-forms}\} \otimes 
\End(\cE)\,, 
$$
that give the structure of a \emph{holomorphic} bundle to a
$C^\infty$-bundle $\cE$ over $X$, provided they satisfy
\begin{equation}
\debar_A^2 = 0 \,. \label{dbar2}
\end{equation}
The \emph{holomorphic} Chern-Simons functional 
\begin{equation}
\CS_\textup{holo}(A) = \int_M  \tr \left( A \wedge \debar A + 
\tfrac{2}{3}  \, A \wedge A \wedge A \right) \wedge \Omega^3 \,,  \label{hCS}
\end{equation}
where $\Omega^3$ is the holomorphic 3-form on $X$, the normalization
of which is of no importance here, similarly satisfies 
\begin{equation}
\nabla \CS_\textup{holo} = 0 \quad \Leftrightarrow \quad  \debar_A^2= 0 \,. \label{critCSh}
\end{equation}
As reflected in the title of \cites{ThCass}, the DT theory was seen
originally as an analog of the 
Casson invariant in algebraic or complex
geometry. As the scope of DT theory 
broadened, there are now both conceptual and technical advantages to
seeing Casson invariants as a particular instance of DT-like counts.

\subsubsection{}\label{s_CSknots} 

Natural observables in CS theory are Wilson lines, that is, holonomies
of $A$ along a curve $K\subset M$, which is a link in $M$, for
instance a knot, whence the notation. A very influential set of ideas 
that originated with Witten \cite{WittCS}, equates those with the
  counts of \emph{open} pseudoholomorphic surfaces 
$$
C \subset T^*M 
$$
ending on the conormal $L_K = T^\perp K \subset T^*M$ to the link $K$, which is a
Lagrangian submanifold of $T^*M$. In many cases, these should further
correspond to counts of complete $C$ in some associated algebraic
varieties, see \cites{GopVafa,OogVafa}. This line of inquiry was the principal
inspiration for the topological vertex conjecture of \cite{AKMV}, see 
Section \ref{s_top_vertex}, 
which in turn ignited a lot of research in DT theory.

\subsubsection{}
There is another very important inspiration that DT theory draws from
CS theory and it is about the role of \emph{representation
  theory} in geometry.  

The action \eqref{CS} does not involve a metric on $M$ which makes CS
theory topological. It can thus be understood in terms of a small
number of standard pieces and representation theory of affine Lie
algebras and/or quantum groups associated to $G$ is fundamental 
for describing those, see e.g.\ \cites{BakKir,Kohno}. 

In DT theory, one similarly desires to describe the standard
pieces of the theory  in terms of (geometric) representation theory, and we
will get a sense how this works in Section \ref{s_actual}, see 
\cites{PCMI,slc} for more.

\subsection{Nekrasov theory}\label{s_NT}

While Nekrasov's theory may be classified as an equivariant
version of Donaldson's invariants of smooth 4-manifolds, its
distinctive goals and technical tools had a very significant influence
on the development of the DT theory. 
% \footnote{For instance, various
% correspondences between DT theory and other enumerative problems were
% from the very beginning phrased as fully equivariant statements. This
% is very helpful both for the development of the theory and for
% practical checks.}. 

\subsubsection{}
Let $A$ be a connection on a trivial rank $r>1$ bundle over
$M=\R^4$, which is the Euclidean version of the Minkowski space-time
of our everyday experience. If finite, the Yang-Mills energy 
$$
\YM(A) = \int_M \| F_\A\|^2 
$$
of $A$ is bounded below by a topological invariant 
\begin{equation}
\int_M \tr F\wedge F  = 8 \pi^2 c_2(F) \le \YM(A) \label{c2int}
\end{equation}
with an equality if $A$ is an \emph{instanton}, also known as an
anti-self-dual connection. The integral \eqref{c2int} agrees with 
the topologically defined Chern class $c_2$ for an extension of
 $A$ to $S^4 \supset \R^4$ as in \cite{Uhl}. 
Modulo gauge transformations that are trivial at $\infty \in S^4$,
instantons are parametrized by a smooth manifold $\cM$ of real 
dimension $4 r c_2$. The integrals over $\cM$ may be interpreted 
as (very!) approximate answers in Euclidean Yang-Mills theory or as
exact answers with enough supersymmetry and twists. This is the physical 
interpretation of Donaldson's theory of integration over a certain
compactification of $\cM$ for general $4$-manifold, see 
\cites{WitTQFT, WittenDon}.  

\subsubsection{}\label{s_torsion_free}

By a theorem of Donaldson \cite{DonGIT}, $\cM$ is also the moduli spaces 
of holomorphic bundles on $X=\C^2 \cong \R^4$ that are
trivial at infinity in the following sense. Pick an embedding into a
projective surface 
$$
X \subset \overline{X}, \quad \textup{e.g.} \quad \overline{X} = \bP^2
$$
and consider holomorphic bundles on $\overline{X}$ together with 
a trivialization 
$$
\phi: \cF|_{D} \xrightarrow{\,\sim\,} \cO_D^{\oplus r} \,, \quad 
D=\overline{X} \setminus X \,, 
$$
known as a \emph{framing} of $\cF$. 
The group $GL(r) \cong \Aut \cO_D^{\oplus r}$ acts on $\phi$, and this
is the same action as the action of constant gauge transformations on
instantons.

In one sentence, Nekrasov theory \cites{NekInst} may be described as 
$G$-equivariant
integration, in cohomology or K-theory,
 over a certain partial compactification 
$$
\cM \subset \overline{\cM}_r = \{ \textup{framed torsion-free sheaves
on $X$ of rank $r$}\}  \,. 
$$
Here 
\begin{equation}
G = \Aut(X) \times GL(r) \owns 
\begin{pmatrix}
  t_1 \\
& t_2 
\end{pmatrix} \times
\begin{pmatrix}
  a_1 \\
& \ddots \\
& & a_r 
\end{pmatrix} \,, 
 \label{GTA}
\end{equation}
where we have introduced an elelement in the maximal torus of $G$. 
In K-theory, Nekrasov counts are functions of $t$ and $a$, while in
cohomology they are functions on the corresponding Lie algebra. 

Very importantly, equivariant variables
are assigned a direct physical meaning in Nekrasov theory. The variables $a_i$ are coordinates on the
moduli space of vacua of the theory. The variables $t_i$, originally
used as cutoff variables to be turned off $t_i\to 1$ eventually
\cites{NOSW,NY1,NY2,NY3,OkICM,OkAMS}, were understood to be even more important as the theory was 
developed further, see \cites{Nikqq1, Nikqq2, Nikqq3, NikPes, NikPesSam}. 

\subsubsection{}
Equivariant localization is a heavy-duty all-purpose tool to do
equivariant counts in terms of the geometry of the fixed locus for 
a maximal torus of $G$. We will get some sense of how it works below,
but  see  e.g.\ \cites{CG} for an excellent 
introduction. 

One can first take the 
fixed locus with respect to the action of $a_i$ and get 
$$
{\overline{\cM}_r}^a = \left( \overline{\cM}_1 \right)^{\times r}
$$
with
\begin{equation}
\overline{\cM}_1  = \Hilb(X,\textup{points}) \supset \cM_1 = 
\{ \cO_X \} = \pt\,. \label{cM_1}
\end{equation}
It is amusing to notice that while instantons are a hallmark of 
nonabelian gauge theories, the integration over them is captured by 
point-like abelian defects as in \eqref{cM_1}. Similarly, I believe 
higher rank DT invariants of 3-folds 
contains substantially the same information
as curve counts, as already discussed in Section
\ref{s_DT_higher_rank}. 

\subsubsection{}\label{s_partitions}

Now 
$$
\Hilb(X,n) = \{ \textup{ideals $I \subset \C[x_1,x_2]$ of codimension
  $n$}\} 
$$
and $t=(t_1,t_2)$ acts on it by 
$$
f (x) \mapsto f(t^{-1} \cdot x) \,. 
$$
If $I$ is fixed then it has to be spanned by the eigenvectors
$x_1^{m_1} x_2^{m_2}$, whence 
\begin{align}
\Hilb(X,n)^{t} &= \{ \textup{monomial ideals $I \subset \C[x_1,x_2]$ of codimension
  $n$}\} \notag \\
&\cong \{ \textup{partitions of $n$}\}  \,, \label{partn} 
\end{align}
where the correspondence between monomial ideals and partitions is best
explained by a picture, see Figure \ref{f_part1}. 
\begin{figure}[!htbp]
  \centering
   \includegraphics[scale=1]{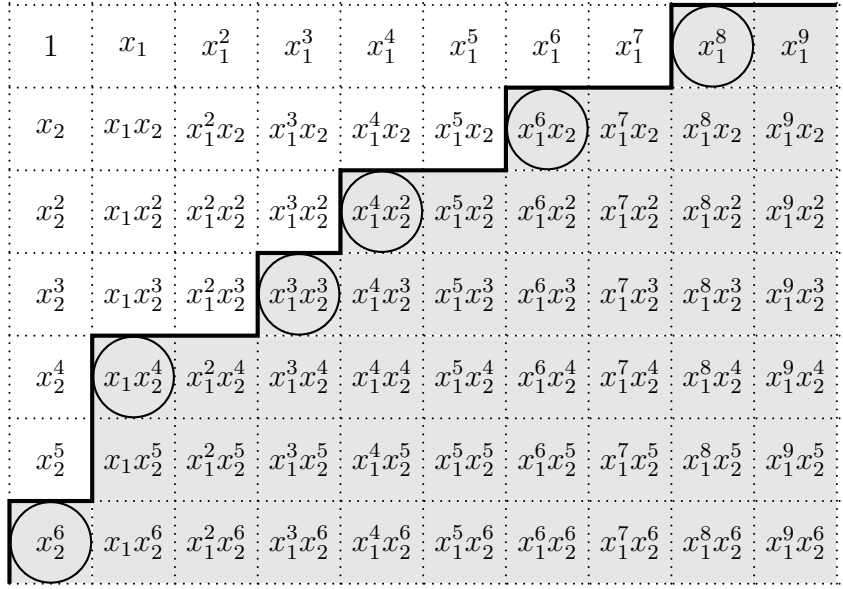}
 \caption{The ideal $I\subset\cO_X$ generated by monomials 
$x_1^8,x_1^6 x_2,\dots,x_2^6$ corresponds to the 
partition $\lambda=(8,6,4,3,1,1)$ of the number $n=23$. The squares in
the diagram of $\lambda$ correspond to a basis of $\cO_X/I$.}
  \label{f_part1}
\end{figure}

\subsubsection{}

For future reference, we point out that 
\begin{itemize}
\item[---] there is parallel match between $d$-dimensional partitions
  of $n$ and the fixed points $\Hilb(\A^d,n)^{\bT}$ of a maximal torus 
$\bT\subset GL(d)$, and 
\item[---] there is a very similar combinatorics for the torus fixed
  points of $\Hilb(X,\textup{curves})$ for a toric variety $X$, see 
Figure \ref{f_part2} and Section \ref{s_toric} below. 
\end{itemize}
\begin{figure}[!htbp]
  \centering
   \includegraphics[scale=1]{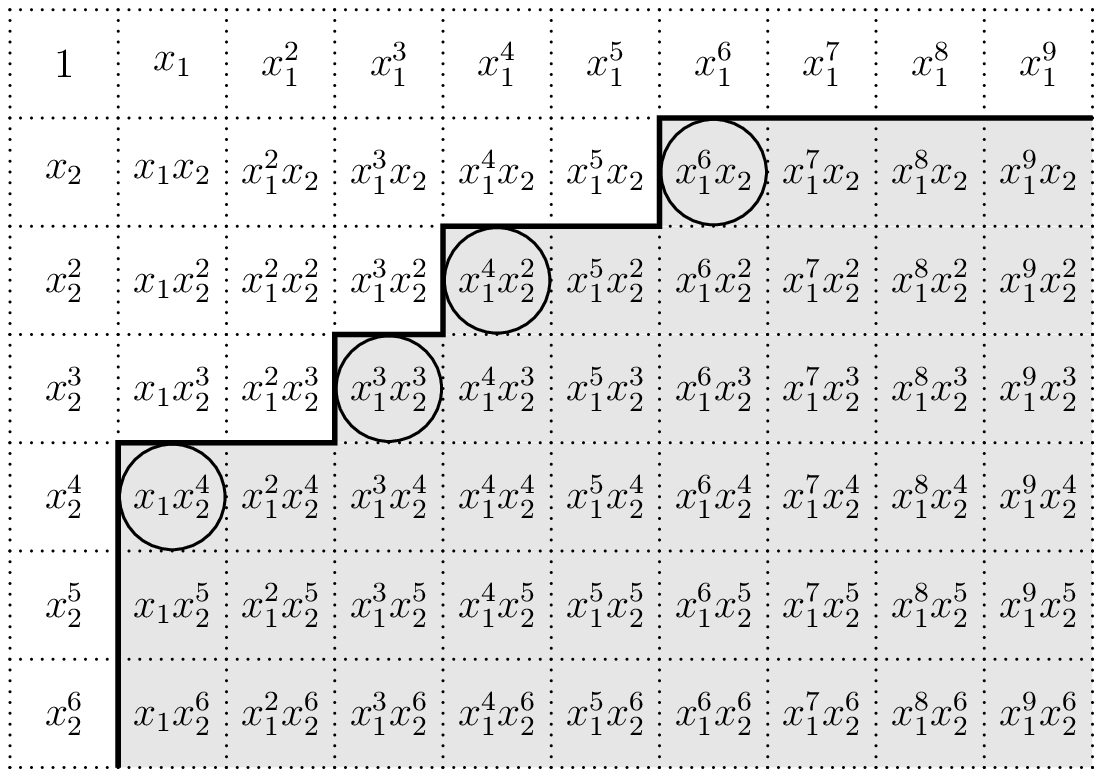}
 \caption{A monomial ideal $I\subset\cO_X$ like in Figure 
\ref{f_part1} need not be of finite 
codimension. This is important when we consider the Hilbert schemes of
curves. }
  \label{f_part2}
\end{figure}

\subsubsection{}

If $\cM$ is a smooth algebraic variety with the action of a torus
$\bT$ then localization formula for $\bT$-equivariant coherent sheaves
on $\cM$ reads
$$
\chi(\cM,\cF) = \chi(\cM^\bT , \cF\big|_{\cM^\bT}\otimes \bSd
N_{\cM/{\cM^\bT}}^\vee )  \in K_\bT(\pt) \left[\frac{1}{1-t^\nu}
\right] \subset \Q(\bT) 
$$
where $t^\nu$ are the weights of the normal bundle $N_{\cM/{\cM^\bT}}$
and $\bSd$ denotes the symmetric algebra. See e.g.\ Chapter 5 in 
\cites{CG} for an introduction. 

For example, if $p\in \cM^\bT $ is an
isolated point and $T_p \cM = \sum t^{\nu_i}$ as a $\bT$-module then 
we get the following contribution 
\begin{equation}
 \bSd
N_{\cM/p}^\vee = \prod \frac{1}{1-t^{-\nu_i}}  \label{bSdN}
\end{equation}
of the normal directions. The product \eqref{bSdN} is the character of
the $\bT$-action on
functions on the formal neighborhood of $p$ in $\cM$. 

\subsubsection{}\label{s_rtuple}

Here we are in the case when all fixed points are isolated and so 
$\chi(\cM_r^\bT)$ is a finite sum over $r$-tuples of partitions. To
compute the character of the tangent space to $\cM$ at the 
point 
\begin{equation}
\cF = I_{\lambda^{(1)}} \oplus \dots \oplus
I_{\lambda^{(r)}} \label{cFsum} 
\end{equation}
we use the modular interpretation of $\cM$. Very generally, tangent
spaces to moduli of sheaves are identified with $\Ext^{1}$ groups, and
concretely 
\begin{align}
T_{\cF} \cM & = \Ext^{1}_{\overline{X}}(\cF,\cF(-D)) \label{TExt1}\\ 
   & = \bigoplus_{i,j=1}^r a_j/a_i\, \otimes
     \Ext^{1}_{\overline{X}}
(I_{\lambda^{(i)}}, I_{\lambda^{(j)}}(-D)) \notag  \,. 
\end{align}
Note, in particular, that this is  is sesquilinear in the summands of
\eqref{cFsum}.

It follows that \eqref{bSdN} becomes a product of
pairwise
interactions between the partitions 
\begin{equation}
 \bSd
T_{\cF}^\vee \cM
= \prod_{i,j} \bEN(\lambda^{(i)},
\lambda^{(j)}, a_j/a_i) 
)^{-1} \label{interbE}
\end{equation}
where 
\begin{equation}
\bEN(\lambda,\mu,u) = \prod_{w \in \textup{ weights of } 
u \otimes \Ext^{1}_{\overline{X}}
(I_{\lambda}, I_{\mu}(-D))}
 (1-w^{-1}) \,. \label{bEu}
 \end{equation}
It is convenient to use the following combinatorial language to
compute the concrete form of this interaction \eqref{bEu}. 

\subsubsection{}\label{s_bG} 

For a diagram $\lambda$, we consider the following generating function 
\begin{equation}
\bG_\lambda = \chi_{\C^2} (\cO/I_\lambda) = \sum_{x_i^a x_j^b \notin
  I_\lambda} t_1^{-a} t_2^{-b} \,. \label{bV2}
\end{equation}
The terms in this sum correspond to the boxes $\square = (a+1,b+1)$ in
the diagram $\lambda$. We define the arm-length and the leg-length of 
a box $\square = (j,i)$ by 
$$
a_\lambda(\square) = \lambda_i - j \,, \quad l_\lambda(\square) =
\lambda'_j -i \,,
$$
where $\lambda'$ denotes the transposed diagram. Note that these numbers
will be negative for $\square\notin \lambda$. 

\begin{Lemma}\label{lemt1} 
The character in \eqref{bEu} is given  by 
\begin{align}
  \Ext^{1}_{\overline{X}}
(I_{\lambda}, I_{\mu}(-D)) & =  \bG_\mu + t_1 t_2 \, 
\overline{\bG}_\lambda - 
(1-t_1)(1-t_2) \, \bG_\mu \overline{\bG}_\lambda \,, 
\label{Ext1f1}\\ 
& = \sum_{\square\in\mu} t_1^{-a_\mu(\square)}
  t_2^{l_\lambda(\square)+1} + \sum_{\square\in\lambda} t_1^{a_\lambda(\square)+1}
  t_2^{-l_\mu(\square)} \,. \label{Ext1f2}
\end{align}
In particular, there are $|\lambda|+|\mu|$ factors in \eqref{bEu}. 
\end{Lemma}

\noindent 
Here $\overline{t^{d}} = t^{-d}$ denotes the usual duality for
representations and characters.

\subsubsection{}\label{s_Extchichi}

Since this is a very typical computation in the subject, 
we do it here explicitly. 

\begin{proof}
 To compute the character of this $\Ext^1$ we note that
\begin{alignat}{2} \notag 
\Ext^{1}_{\overline{X}}
(I_{\lambda}, I_{\mu}(-D)) &= - \chi_{\overline{X}} (I_{\lambda},
I_{\mu}(-D))\,,  \qquad \quad   &\textup{because $\Ext^0=\Ext^2=0$,}&\\ 
&= \chi_{\C^2} (\cO,\cO) - \chi_{\C^2} (I_{\lambda},
I_{\mu})\,,  &\textup{by localization.} \label{ext1chichi} &
\end{alignat}
Further note that $\chi_{\C^2} (\cF, \cG)$ is a sesquilinear form on 
$$
K_\bT(\C^2) \cong K_\bT(\pt)
$$
normalized by 
$$
\chi_{\C^2} (\cO,\cO) =
\frac{1}{(1-t_1^{-1})(1-t_2^{-1})} \,. 
$$
We have 
$$
[\cO/I_\lambda] = \frac{\chi_{\C^2}(\cO/I_\lambda)}{\chi_{\C^2}(\cO)} 
\, [\cO] = \bG_\lambda (1-t_1^{-1})(1-t_2^{-1})) [\cO] \in K_\bT(\C^2) 
$$
which gives \eqref{Ext1f1}. For \eqref{Ext1f2} see e.g.\ Lemma 3 in 
\cites{CarOk}. 
\end{proof}

\subsubsection{}

The arms-and-legs combinatorics of the interaction \eqref{bEu} is
typical in Nekrasov theory which extends to higher rank, that is, to
many partitions $\{\lambda^{(i)}\}$ the combinatorics of Macdonald
symmetric functions. 

The natural generality to study Nekrasov counts is when the 
gauge group has several factors (not unlike
what happens in the standard model), so that 
\begin{equation}
\Mbar = \prod \Mbar_{r_i} \supset \left\{ \textup{${\textstyle \prod}
    U(r_i)$-instantons on $\R^4$} \right\}  \,,
\label{product_gauge}
\end{equation}
and we define 
\begin{align}
\bZ_\textup{preliminary} &= \chi(\overline{\cM}, {\textstyle \prod} z_i^{c_2(\cF_i)} 
\Euler(\textup{matter}) )\label{matter} \\
& =\sum_{\textup{$r$-tuple of parititions}}  z^{\textup{\# of boxes}} \!\!\!\!\!\!\!\!
\prod_{ \substack{\eta,\nu\in \, \textup{$r$-tuple}
\\ \textup{interactions with mass $u_k$}}} \bEN(\eta,\nu,u_k)^{\pm 1}  \label{bEupm1}
\end{align}
where $r =\sum r_i$ is the total rank, 
$$
z^{\textup{\# of boxes}}  =  {\textstyle \prod} z_i^{c_2(\cF_i)}  = 
z_1^{\sum |\lambda^{(i)}|} z_2^{\sum |\mu^{(j)}|}  \cdots  \,. 
$$
The matter in \eqref{matter} is a bunch of
  fermions in representations of the form 
$(\C^{r_i})^* \otimes \C^{r_j}$ of the gauge group in 
\eqref{product_gauge} and mathematically described by the vector
bundle 
\begin{equation}
\xymatrix{
\Ext^{1}_{\overline{X}}
(\cF_i, \cF_j(-D)) \ar@{^{(}->}[rr] \ar[d] &&  
%\textup{$(\C^{r_i})^* \otimes
% \C^{r_j}$-matter bundle} 
\Psi_{i,j}
\ar[d]\\
\{(\cF_i, \cF_j) \}\ar@{^{(}->}[rr] &&\Mbar_i \times \Mbar_j 
}
\label{matter_bundle}
\end{equation}
of rank $c_2(\cF_i) + c_2(\cF_j)$. This bundle is given an 
additional equivariant weight $m$ with respect to some bigger 
torus. This weight gives the mass to the fermion and
contributes
$$
\Euler(m\otimes \Psi_{ij}) = \sum_k (-m)^{-k} \Lambda^k \Psi_{ij}^\vee
$$
to the integrand in \eqref{matter} and $\bEN(\eta,\nu,m)$ to the
product in \eqref{bEupm1}. 

The terms with $\bEN(\lambda,\mu,\dots)$ in denominator in
\eqref{bEupm1} in come from
\eqref{interbE}. The simple mathematical 
fact that they appear in the denominator 
may be explained physically as the interaction of $\eta$ and $\nu$ via
a gauge \emph{boson} of the corresponding gauge group. The 
equivariant weight $a_j/a_i$ in \eqref{interbE} corresponds to the
mass of that gauge boson.

Note that the case of matter in the fundamental representation
$\C^{r_i}$, or its dual, for one of the gauge group factors is
contained in the previous discussion as a special case. To find it, take $\cF_i$ or 
$\cF_j$ trivial in \eqref{matter_bundle}. 

\subsubsection{}
The true Nekrasov partition function differs from \eqref{matter} in
two respects. First, there is a $z$-independent prefactor, interpreted
as perturbative contributions to the partition function. It is
obviously very important, but will be left out from this discussion. 

Second, the partition functions like \eqref{matter} come from susy
quantum mechanics, that is, indices of suitable Dirac operators, on the
moduli spaces $\Mbar$. On a K\"ahler manifold, this Dirac operator 
is the $\debar$ operator twisted by the square root $\cK^{1/2}_{\Mbar}$
of the canonical bundle.  This will be also a very important feature
of K-theoretic DT counts. This square root twist propagates in 
formulas by 
\begin{equation}
  \label{Z_Nekr}
 \bZ_\textup{Nekrasov} = \bZ_\textup{perturbative}\, \bZ_\textup{preliminary}
  \big|_{\bEN\mapsto \widehat{\bEN}}
\end{equation}
where 
\begin{equation}
\widehat{\bEN}(\lambda,\mu,u) = \prod_{\textup{same $w$}}
 (w^{1/2}-w^{-1/2}) \,. \label{bEh}
 \end{equation}

\subsubsection{}

The properties of \eqref{Z_Nekr} are very deep and rich, see e.g.\ 
\cites{Nikqq1, Nikqq2, Nikqq3, NikPes, NikPesSam}.  They place
a lower bound on the complexity and richness 
of DT counts as
Nekrasov counts can be found within DT counts, sometimes in a rather
nontrivial way, see e.g.\ Section \ref{s_engin} below.

The main conjecture of \cites{NekInst} was about the limit 
$t_1,t_2\to 1$ and its relation to the geometry of Seiberg-Witten
curves, see \cites{NOSW,NY1,NY2,NY3,OkICM,OkAMS}. This limit is
perhaps best understood from the 
3-dimensional perspective and the 3-dimensional boxcounting 
interpretation of Nekrasov functions. The whole Seiberg-Witten curve
can be clearly seen in the limit shape for boxcounting 
problem\footnote{Also, the variational principle for limit 
shapes is
simpler in the 3-dimensional setting, as it may be taken to describe
  a random surface with local interaction, see e.g.\ \cite{OkAMS}. This is easier than the
  variational problem for random partitions solved in \cite{NOSW}
  because local rules for surfaces generate nonlocal interactions for
  the partitions that appear as slices of a random surface.}.

Among other things, the functions \eqref{Z_Nekr} generalize and 
discretize many important integrals over $r$-tuples of
Hermitian matrices, in the same way as summations over partitions may
be seen as a discrete analog of a random matrix integral\footnote{
See e.g.\ \cites{OkUses} for a lengthy discussion of such comparisons.}. 

\subsubsection{}

The combinatorics surrounding Macdonald symmetric functions is full of
identities of the general form 
$$
\sum \prod \frac{1-t^{\mu}}{1-t^{\nu}} = \textup{another expression of 
the same breed} \,,
$$
where $t=(t_1,t_2,\dots)$, and, as a rule, one gets a lot of insight
and mileage out of interpreting such identities as statements about 
equivariant K-theory of $\Hilb(\C^2)$ and related spaces. In Nekrasov
theory one finds a multitude of higher rank generalization of such
identities. 

Perhaps as a meta principle one could propose
 the following: \textsl{every problem
involving partitions, and also Schur functions for some special values
of the parameters, is
really a problem in DT theory and is best approached as such}.

\section{The GW/DT correspondence} 

\subsection{Main features} 

\subsubsection{}

At this point, the reader will be hardly surprised to learn there is
a correspondence between DT and GW counts for 3-folds, but the exact
details of this match may be surprising. Most importantly,
there is no way to match the integrals over the individual moduli spaces in
DT and GW theories.

Indeed, even if we allow disconnected curves
(which can have negative genus), we have 
\begin{equation}
g \ll 0 , \,\, \textup{$d$ fixed} \quad \Rightarrow \quad
\Mbar_{g,n}(X,d)^\bullet = \varnothing \,, \label{GWempt}
\end{equation}
where bullet indicated that we allow disconnected domains $C$ as long
as $f$ is not constant on any of them. Similarly, 
\begin{equation}
\chi \ll 0 , \,\, \textup{$d$ fixed} \quad \Rightarrow \quad 
\Hilb(X,d,\chi) = \varnothing \,. \label{DTempt}
\end{equation}
If we try to match the discrete invariants by the usual formula
$$
\chi =1 - g\,,
$$
then we see a
serious disagreement between \eqref{GWempt} and 
\eqref{DTempt}. 

If fact, the correspondence will be an equality of generating
functions over $\chi$ and $g$. The two generating 
functions will identified not as formal power
series but as analytic functions 
after a certain change of variables. This means that, a
priori, to reconstruct one integral on one side infinitely many
integrals on the other side are needed. In reality, this is much more
effective because one of the functions is conjectured to be
 \emph{rational}. 

\subsubsection{} 
A nice geometric way to remember the degree of the curve and 
forget its genus is to consider the maps to the Chow variety of 
$1$-dimensional cycles in $X$
\begin{equation}
  \label{toChow}
  \xymatrix{ \PT(X) \ar[dr]_{\pi_{\PT\to\Chow}\quad } &&  
\Mbar(X)^\bullet  \ar[dl]^{\quad\pi_{\Mbar\to\Chow}} \\
& \Chow(X)} \,, 
\end{equation}
where $\PT(X)$ stands for the moduli space of stable pairs. One can
also 
put the Hilbert scheme of curves, or other DT moduli spaces in its
place, with only minor changes to the correspondence. 
The maps in \eqref{toChow} are proper once we fix
$\chi(\cF)$ or $g(C)$, respectively. This means that 
\begin{equation}
\bZ_{\PT\to\Chow} = \pi_{\PT\to\Chow,*}  
\left((-z)^{\chi(\cF)} \, \left[\PT\right]_\vir \right)
\in H_{2 \virdim} (\Chow(X),\Z) ((z))\label{ZPTChow}
\end{equation}
is a well-defined as a formal Laurent series. Here 
$$
\virdim = c_1(X) \cdot \textup{cycle} 
$$
 is a locally constant function on $\Chow(X)$. Similarly, we
define 
\begin{equation}
\bZ_{\Mbar\to\Chow} = \pi_{\Mbar\to\Chow,*}  
\left(u^{2g-2} \, \left[\Mbar^\bullet \right]_\vir \right)
\in H_{2 \virdim} (\Chow(X),\Q) ((u^2)) \,. \label{ZMbChow}
\end{equation}
The $\Q$-coefficients appear here because the virtual cycle of 
the orbifold $\Mbar^\bullet$ are only defined with rational
coefficients. 

The GW/DT correspondence, in its basic form, is the following 

\begin{Conjecture}[\cites{MNOP1,MNOP2}]\label{cr} 
 The series \eqref{ZPTChow} is an expansion of a rational function
in $z$ with poles at roots of unity. 
\end{Conjecture}

\begin{Conjecture}[\cites{MNOP1,MNOP2}]\label{c1} 
We have 
 \begin{equation}
   \label{GWDT}
   z^{-\frac{\virdim}{2}} \, \bZ_{\PT\to\Chow}  = (-iu)^{\virdim} \, \bZ_{\Mbar\to\Chow}
 \end{equation}
after the change of variables 
\begin{equation}
z = e^{iu}\,. \label{zeu}
\end{equation}
\end{Conjecture}

\medskip 

\noindent 
Originally, the conjecture was stated for numerical Hilbert scheme counts, the formulation here makes use of later
improvements. Numerical counts are obtained from \eqref{GWDT}
by pairing with cohomology
classes that record incidence of cycles. 

\subsubsection{}

Conjecture \ref{c1} is proven for all toric varieties in \cites{MOOP} and,
for numerical counts, for complete intersections in the products of
projective spaces in \cites{PP5}. In fact, Pandharipande and Pixton prove
a finer correspondence that includes descendent insertions. Early 
discussion of such descendent correspondence may be found in \cites{MNOP2},
see also \cites{OOP}. Tracing the GW/DT correspondence through degenerations of
the form \eqref{degen} and equivalences between different kinds of
boundary conditions play a key role in \cites{PP1,PP2,PP3,PP4,PP5}.

\subsubsection{}\label{s_parity}

Note, in particular, that \eqref{GWDT} must have the same parity with respect to 
$$
(z,u) \mapsto (z^{-1},-u) 
$$
as the parity of the virtual dimension. This parity compensates the 
Galois action $i \mapsto -i$ on the identification \eqref{zeu} between 
series with rational coefficients. 

\subsubsection{}

The following is a heuristic analytic argument for the rationality of
\eqref{ZPTChow} assuming \eqref{GWDT} is an equality of analytic
functions on some common domain of analyticity. The series
\eqref{ZPTChow} is a series in $z$ with \emph{integer}
coefficients and, at least in specific instances, it is easy to see it
converges for $|z|<1$. 

A classical theorem of F.~Carlson \cites{FritzC,Remm} then implies 
\begin{itemize}
\item[---] either it is a rational function, or 
\item[---] the unit circle $|z|=1$ is a natural boundary for it. 
\end{itemize}
In the latter case, it cannot be meromorphic in any neighborhood 
of the points $z=1$, whence the conclusion. 

% \subsubsection{} 

% The rationality in Conjecture \ref{cr} and the parity discussed in
% Section \ref{s_parity} remind many people of Weil
% conjectures. Conjecture \ref{c1} could then be interpreted as a
% geometric formula for Taylor coefficients at the central point of the 
% functional equation, perhaps to be compared to \cites{}. 

% I am not aware of any actual connections along these lines, and the
% apparent structural similarity may be explained as follows. In section
% \ref{s_M} we will discuss further conjectures that equate the DT
% counts with Euler characteristics of certain $z$-equivariant 
% coherent sheaves, where $\Ct\owns z$ acts algebraically on certain
% moduli spaces. This gives a control of denominators, explains parity, 
% but does not say anything special about the numerators of the 
% rational functions

\subsection{Relative correspondence}

\subsubsection{} 

Correspondence of relative boundary conditions in GW and DT theories takes a remarkably
simple form. We will interpret relative partitions functions as vectors in
the Fock spaces \eqref{ZXD3} and \eqref{bcSym}, respectively. 
The monomial prefactors in \eqref{GWDT} may be absorbed in a change
of variables of the form 
$$
\bZ(Q z^{-c_1(X)/2},z) = z^{-\virdim/2} \bZ(Q,z)  
$$
and we will assume that this has already been done. 
To match the boundary conditions, we need a map 
\begin{equation}
\bSd \left(\Hd(D) \otimes t \Q[[t]]\right) \otimes \F 
\xrightarrow{\quad \sim \quad} \Hd (\Hilb(D)) \otimes \F
\,, \label{GWDTrel}
\end{equation}
where $\F$ denotes the functions of $u$ and $z$, and such 
a map is uniquely determined by where it sends the operators of
multiplication by $\gamma 
\otimes t^k$ for $\gamma\in \Hd(D)$ in the symmetric algebra. 

\begin{Conjecture}[\cites{MNOP2}]\label{c2} 
The relative correspondence map \eqref{GWDTrel} sends the 
multiplication operator by 
$u^{1-k} \gamma \otimes t^k$ to the Nakajima 
creation correspondence that adds a length $k$ subscheme supported on
the cycle Poincar\'e dual to $\gamma$. 
\end{Conjecture}

% \smallskip 
% \noindent 
% The need for the twist by $u^{1-k}$ here may be traced back to the 
% following basic fact. In GW theory, we weight different genera by 
% $u^{-\chi_\textup{top}(C)}$ and 
% $$
% \chi_\textup{top}(C) = \chi_\textup{top}(C\setminus D) + \# (C \cap
% D)\,,
% $$
% where $\# (C \cap D)$ is the cardinality of the (reduced)
% intersection. In DT theory, by contrast, we count the intersection
% with D with multiplicity, as in e.g.\ \eqref{chiglDT}. 
% \task{fix this} 

\subsubsection{}
A weak form of Conjecture \ref{c2}, a correspondence between certain
\emph{capped} counts was established for toric varieties in \cites{MOOP}. 
Much stronger results were obtained in \cites{PP4,PP5}.

\subsection{Example: GW theory of curves} 

\subsubsection{}

Let $B$ be a smooth curve of some genus. A stable map 
$$
f: C \to B
$$
is a branched cover of $B$ on some components of $C$ and constant on
other components of $C$, see Figure \ref{f_Hurw}. 

 At the dawn of
representation theory, A.~Hurwitz sorted out the enumeration of
degree $d$ branched covers in terms of the characters of the symmetric
group $S(d)$, see e.g.\ \cites{Jones} for a survey. It is difficult to
improve on this classical treatment, except that there is a certain
combinatorial complexity to characters of $S(d)$ that kept generations
of researchers busy. 

The whole GW theory of $B$ is a certain mixture of Hurwitz theory with the
contributions of collapsed components of $C$ which have the form 
\begin{equation}
\lan \prod \psi_i^{m_i} \, \lambda_k 
\ran := \int_{\Mbar_{g,n}} c_k (\textup{Hodge}) \prod_{i=1}^n
c_1(T^*_{p_i})^{m_i} \,,
\label{psilam}
\end{equation}
where the integral is over the Deligne-Mumford moduli space of genus
$g$ stable curves $C$ with $n$ marked points $\{p_i\}$ and $T^*_{p_i}$
denotes the line bundle with fiber $T^*_{p_i} C$. 

\begin{figure}[!htbp]
  \centering
   \includegraphics[scale=1.4]{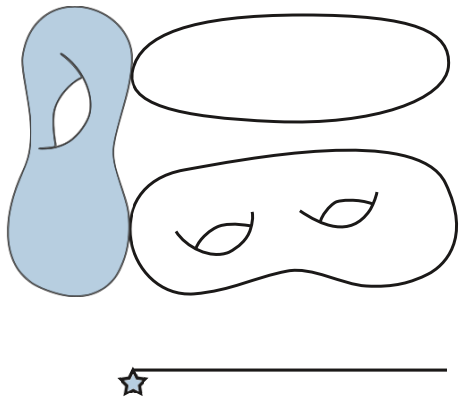}
 \caption{Some components of the source curve $C$ are branched covers
   of the target curve $B$, while other components of $C$ are
   collapsed by the map $f$.}
\label{f_Hurw}
\end{figure}

\subsubsection{}
The GW theory of curves was worked out explicitly in \cites{OP1,OP2,OP3} and it
turns out 
\begin{itemize}
\item[---] it is much simpler combinatorially than the Hurwitz theory, 
\item[---] both the $S(d)$-characters and the integrals \eqref{psilam}
  may be deduced from it, and this gives much better results than
  previously known. 
\end{itemize}
The answers in GW theory of B are some explicit sums over partitions
of $d$, which are, of course, unavoidable as long as $S(d)$ is
around. 

One may view these sums as finite discrete analogs of 
random matrix integrals that played a very important role in Witten's
pioneering thinking \cites{WittenDM} about 2-dimensional quantum gravity,
matrix models, and integrals \eqref{psilam} with $k=0$. See \cites{OkUses}
for more on random matrices versus 
random partitions. 

\subsubsection{}
These sums over partitions now find a completely transparent
interpretation via GW/DT correspondence if we take 
$$
X = B \times \A^2_{(t,-t)}
$$
and use the periodicity \eqref{period}. On the DT side, there is also an
analog of Mumford's vanishing: for opposite equivariant weights, 
the virtual class vanishes as soon 
as the map $\cO_X \to \cF$ is not surjective. The only moduli spaces to
consider are then 
$$
\PT(X, d [B]) \cong \Hilb( \A^2_{(t,-t)},d) \,, 
$$
where the isomorphism is induced by pullback under $X \to
\A^2_{(t,-t)}$.  

This reduces everything to the \emph{classical} geometry of $\Hilb(
\A^2)$, which yields sums over partitions by simple localization
like in Section \ref{s_partitions}. 

\subsubsection{}
The generalization of this geometry is the theory of \emph{local
  curves}, that is, threefolds that are total spaces
\begin{equation}
X = 
\begin{matrix}
\cL_1 \oplus \cL_2 \\
\downarrow\\
B
\end{matrix} \,,\label{local_curve}
\end{equation}
of two line bundles on $B$. There is a natural action of $(\Ct)^2$ in
the fibers of $X$ and the equivariant variables 
\begin{equation}
(t_1,t_2) \in \Lie (\Ct)^2 \label{Ct1t2} 
\end{equation}
are important parameters in the theory. 

The GW and the DT sides of the story were worked out in \cites{BrPand} and 
\cites{OP5}, respectively. They now related to the \emph{quantum} cohomology of 
$\Hilb(\A^2_{t_1,t_2})$, computed in \cites{OP4}. This theory now
strictly generalizes the corresponding 
classical story of Jack
polynomials etc. We will have another look at it from an even higher
perspective in Section \ref{s_actual}.

\section{Membranes and sheaves}\label{s_M} 

\subsection{Outline} 

\subsubsection{} 

M-theory is an ambitious vision that ties together many threads of the
modern high-energy physics in a unique 11-dimensional supergravity
theory. Instead of point particles or strings, M-theory contains membranes (M2-branes) with 3-dimensional
worldvolume, see e.g.\ \cites{MMM} for a survey of their properties. 

In a space-time of the form $Z \times S^1$, where $Z$ is a complex
Calabi-Yau 5-fold, there will be supersymmetric M2 branes of the form 
$C \times S^1$ where $C\subset Z$ is a holomorphic curve. The
contribution of those is expected to be an enumerative theory
superficially resembling other curve counting theories. 

A guess for what this theory might look like and how it should be
related to DT theories of 3-folds is the main theme of \cites{MDT}. Among
other things, the conjectures of \cites{MDT} provide a natural description of the rational
function in Conjecture \ref{cr}. They do so by identifying a
generalization of the series 
\eqref{ZPTChow} with a certain equivariant K-theoretic count of M2-branes, in
which the variable $z$ is viewed as acting on $Z$ via 
$$
z \in \Ct_z \subset \Aut(Z,\Omega^5_Z) \,. 
$$
In English, we require that $z$ preserves the Calabi-Yau 5-form
$\Omega^5_Z$. By
localization, such counts are always rational functions with
controlled
poles.

To be able to interpret a multiplicative variable 
$z\in \Ct$ as an equivariant parameter, we
must work in K-theory and the corresponding M2-counts will be 
matched with the K-theoretic analogs
of \eqref{ZPTChow}. Conveniently, the suitable K-theoretic extension of DT 
counts works out very nicely. 

\subsubsection{}

The main difference between counting M2-branes and what we have seen
before is that M-theory lacks a parameter that could keep track of the
genus of $C$. This makes sense if we want a correspondence with DT
counts that reassigns the genus-counting variable $z$. 

M-theory has a 3-form field that generalizes 
connections of gauge theories and determines, via integration over the
worldvolume, the action of an M2-brane. For curve counting, this
produces variables $Q$ that keep track of the degree of $C$. But there
is nothing in M-theory that could similarly couple to the genus of
$C$, or the Betti numbers of the worldvolume $C \times S^1$, hence we
must sum over the genera with no extra weight. 

The only possible solution conclusion is that these sums must be
finite and therefore we require the map 
$$
\MM(Z) \to \Chow(Z) 
$$
to be \emph{proper}, where $\MM(Z)$ is our hypothetical moduli space
of supersymmetric M2-branes. This is very different in flavor from
either GW or DT counts and is achieved in \cites{MDT} by imposing a
certain stability condition. 

\subsubsection{}

A connection with DT counts appears when the fixed locus 
$$
X = Z^{\Ct_z}
$$
has pure dimension $3$. For simplicity, we will assume here that $X$
is connected, otherwise DT theories of different components of $X$
will be talking to each other like the different partitions were
talking to each other in Section \ref{s_rtuple}. The details of that
interaction can be found in Section 3 in \cites{MDT}, but here we skip
them.

Here we can take
\begin{equation}
Z = 
\begin{matrix}
\cL_4 \oplus \cL_5 \\
\downarrow\\
X
\end{matrix} \,,\label{localX}
\end{equation}
where $X$ is an arbitrary smooth quasiprojective $3$-fold, 
$$
\cL_4 \otimes \cL_5 = \cK_X\,, 
$$
and $\Ct_z$ acts with weights $(z,z^{-1})$ in the fibers of
\eqref{localX}. 

The correspondence between DT and M2 counts will be expressed as an 
equality of two elements of $K(\Chow(X)) (\!(z)\!)$,
in parallel to the language of Conjecture \ref{c1}. 

\subsubsection{}

On the DT side, we consider the following analog of \eqref{ZPTChow}
\begin{equation}
\bZ_{K,\PT\to\Chow} = \pi_{\PT\to\Chow,*}  \,
\tO_{\PT} 
\in K(\Chow(X)) (\!(z)\!)\label{ZKPTChow}
\end{equation}
where
\begin{equation}
  \label{deftODT}
  \tO_{\PT} = \textup{prefactor}  \, \, 
 \cO_\vir \otimes \left( \cK_\vir \otimes \det\Hd(\cF\otimes(\cL_4-\cL_5)) \right)^{1/2}
\end{equation}
with 
\begin{align}
   \cO_\vir &= \textup{the virtual structure sheaf the moduli space $\PT(X)$}\,,
\notag  \\
 \cK_\vir & = \textup{its virtual canonical line bundle} \,, \notag \\ 
\det\Hd(\cF\otimes \dots ) &= \textup{a tautological line bundle on
                             $\PT(X)$}\,, \notag \\
\textup{prefactor}  &= (-1)^{(\cL_4,d)+\chi(\cF)} \,
                      z^{-\frac{\textup{vir dim}}2+\chi(\cF)}
                      \,.  \label{prefactor} 
\end{align}
Here the virtual structure sheaf and the virtual canonical line bundle
come out of the general machine of perfect obstruction theories, see
an example below. 
The fact that a certain line bundle has a square 
root\footnote{A more precise statement is that it is a square modulo a
certain fixed line bundle pulled back from $\Chow(X)$, which will also
appear on the membrane side.} uses something specific
about DT moduli spaces. The existence of the required square root is
shown in \cites{MDT}. The Euler characteristics that we compute here
do not depend on the choice of the square root. 

Note that, aside from $ (-1)^{(\cL_4,d)}$, the prefactor is the same
monomial $\pm z^{\dots}$ that appears in the correspondence
\eqref{GWDT}.  The expression \eqref{deftODT} is a slight modification
of the 
$$
\textup{virtual $\hat A$-genus} = \cO_\vir \otimes \cK_\vir^{1/2}
$$
that contains the right dependence on the normal bundle $N_{Z/X}$ and
has
a
well-defined square root. 

Before discussing the membrane side of the story, it is very
instructive to consider an example. 

\subsection{Smooth curves} 

\subsubsection{} 

Let $X$ be of the form \eqref{local_curve}, which means that 
\begin{equation}
Z = 
\begin{matrix}
\cL_1 \oplus \cL_2 \oplus \cL_4 \oplus \cL_5 \\
\downarrow\\
B
\end{matrix} \,,\label{local_curveM}
\end{equation}
and consider curves of minimal degree
$$
d= [B] \in H_2(Z,\Z) \,,
$$ 
where $[B]$ is the class of a section in 
\eqref{local_curveM}. The corresponding component of the 
Chow variety is simply a 
linear space 
$$
\Chow(X,d) = H^0(B,N_{X/B})\,, \quad N_{X/B}=\cL_1 \oplus \cL_2 \,, 
$$
and by the analysis of \eqref{XsF} 
the fibers of $\pi_{\PT\to\Chow}$ are symmetric powers of the 
underlying curves 
\begin{equation} 
\label{pbSdB}
\xymatrix{
  \bSd\ar[d] B 
\ar@{^{(}->}[r]
& \PT(X,d) \ar[d]^{\pi_{\PT\to\Chow}} \\
[B]
\ar@{^{(}->}[r] 
& \Chow(X,d) \,. 
}
\end{equation}
Without loss of generality,  it suffices
to consider the fiber over the curve $B$ itself in \eqref{pbSdB}. 
In any event, all counts can be reduced to this fiber by 
localization with respect to the torus \eqref{Ct1t2},

\subsubsection{} 

First, let us discuss the integrand in \eqref{ZKPTChow} in the special 
case 
\begin{equation}
\cL_1 \otimes \cL_2 = \cK_B\,, \quad \cL_4 = \cL_5 = \cO_B \,. 
\label{assume} 
\end{equation}
In this case, the 3-fold $X$ is already Calabi-Yau and the obstruction
theory of the PT moduli spaces is self-dual. 

The deformation-obstruction theory can be divided into two pieces. 
One, which we call horizontal, described the deformations 
$$
\Def_{\textup{---}} = H^0(B, N_{B/X}) \,, \quad 
\Obs_{\textup{---}}  = H^1(B, N_{B/X}) = \Def_{\textup{---}}^\vee 
$$
of the curve itself and is pulled back from the Chow variety. The vertical piece 
$$
\Def_{|}= \textup{the tangent bundle $T \bSd B$} \,, \quad \Obs_{|} = 
\Def_{|}^\vee \,, 
$$
describes the deformations in the fibers of \eqref{pbSdB}.

\subsubsection{}

Whenever
the obstruction theory is given by a vector bundle $\cE$ on a smooth
variety $M$,
the virtual structure sheaf is cut out by a section $s$ of $\cE$. 
By the Koszul resolution, we have 
$$
\cO_\vir = \bigoplus (-1)^k \Lambda^k \cE^\vee = \bSd ( - \cE^\vee)
\,, 
$$ 
with the same sign conventions for the symmetric algebra as in
\eqref{bcSym}. Also, 
$$
\cK_\vir = \cK_M \otimes \det \cE 
$$
is the determinant of $\Obs-\Def = \cE - TM$. 

It is convenient to define the symmetrized symmetric
algebra by 
\begin{equation}
\bSdh V = (\det V)^{1/2} \, \bSd V = (-1)^{\rk V} \,\, \bSdh V^\vee
\,, \label{bSdhsymm} 
\end{equation}
where the last equality holds if $V$ is odd or in localized
equivariant K-theory. With this notation 
$$
\cO_\vir \otimes \cK_\vir^{1/2} = \cK_M^{1/2} \otimes \bSdh ( -
\cE^\vee) \,. 
$$
The localization of this to a point of $M$ equals 
$$
\bSdh(T^*M- \cE^\vee) = (-1)^{\virdim}\,\,  \bSdh (TM - \cE) \,. 
$$

\subsubsection{}
We now specialize this to the vertical part of the obstruction theory
with 
$$
M = \bS^n B \,, \quad \cE = T^* M \,. 
$$
With the $\pm z^{\dots}$ prefactor in \eqref{prefactor}, we get 
$$
\tO_{\PT,|} \, \Big|_{\bS^n B} = (-z)^{1-g(B)} z^n \sum_p (-1)^p \, 
\Omega^p \, 
\bS^n B  \,. 
$$
The pushforward of this is a combination of Hodge structures of  
$\bSd B$, namely 
\begin{align}
\pi_{\PT\to \Chow,*} \tO_{\PT,|}  &=  (-z)^{1-g(B)} \sum_{n,p,q} z^n
(-1)^{p+q} H^{p,q}(\bS^n B)  \,, \notag \\
&= (-z)^{1-g(B)} \,  \bSd \left(z \sum_{0\le p,q\le 1}
(-1)^{p+q} H^{p,q}(B)   \right)
\label{Hpq} 
\end{align}
where 
$$
H^{p,q}(Y) = H^q(\Omega^p Y)\,.
$$
The equality in \eqref{Hpq} is a classical result that goes
back to Macdonald \cites{MacdS} and says that the Hodge structures of $\bSd B$
are canonically the symmetric algebra of the Hodge structure of $B$
itself.

\subsubsection{}

The argument of $\bSd$ in \eqref{Hpq} can be written as follows 
$$
z (-1)^{p,q} H^{\bullet,\bullet}(B) = \Hd(B,z \cL_4) + \Hd(B, z^{-1}
\cL_5)^\vee 
$$
and so from \eqref{bSdhsymm} we conclude 
\begin{align}
 \pi_{\PT\to \Chow,*} \tO_{\PT,|}   &= \bSdh \Hd(B, z \cL_4 \oplus
  z^{-1} \cL_5)^\vee \notag \\
& = \bSdh \Hd(B, N_{Z/X})^\vee \,, \label{pptsum} 
\end{align}
which is a very remarkable conclusion. 

In English, it says that the integration along the fibers of
\eqref{pbSdB} can be replaced by just allowing the curve $B$ to move
in the 4th and 5th dimensions ! Indeed \eqref{pptsum} is identical to 
$$
\textup{localization of }  \tO_{\PT,\textup{---}} = 
\bSdh \Hd(B, N_{X/B})^\vee
$$
and the combination of the horizontal and vertical part treats all
normal directions to $B$ in $Z$ equally. 

There is a simple 

\begin{Theorem}[\cites{PCMI}]
Formula \eqref{pptsum} is valid without the assumptions
\eqref{assume}. 
\end{Theorem}

\noindent
Its proof is a natural modification of Macdonald's result for twisted
cotangent bundles of $\bS^n B$ that appear in the general case. 
So, the conclusion for any smooth curve $B$ is that its PT theory is
summed up by allowing it to move in the extra dimension of M-theory.

\subsection{General curves}

\subsubsection{} 
Still in the setting of \eqref{localX}, our goal now is to discuss a
general conjectural formula for \eqref{ZKPTChow} in terms of the 
the symmetrized 
virtual structure sheaf 
$$
\tO_{\MM} = 
 \cO_\vir \otimes \cK_\vir^{1/2} 
$$
of the membrane moduli spaces. This formula should generalize 
\eqref{pptsum} and therefore should involve $\Ct_z$-localization. 

\subsubsection{} 

Consider the diagram of proper maps 
\begin{equation}
\xymatrix{
\MM(Z)^{\Ct_z} \ar[d] \ar[r]^\iota & \MM(Z) \ar[d] \\
\Chow(X) \ar[r] & \Chow(Z) \,. 
}
\end{equation}
By localization, the map $\iota_*$ on equivariant K-theories is an 
isomorphism after inverting $1-t z^n$, where $t$ is weight of
$\Aut X$ and $n\ne 0$. Thus, we can define 
\begin{equation}
\bZ_{K,\MM\to\Chow} = \pi_{\MM\to\Chow,*}  \, (\iota_*)^{-1} \,
\tO_{\MM} 
\in K_{\Aut(X)}(\Chow(X)) \left[z^{\pm 1},\frac{1}{1-t z^n} 
\right] \,. 
\label{ZKMMChow}
\end{equation}

\subsubsection{} 
By design, the membrane moduli space parametrize \emph{connected}
membranes and thus there is a certain exponentiation to go from the
membrane counts to the DT counts. 

The Chow variety is an algebraic semigroup, where the addition maps 
$$
(+)_n : \Chow(X)^n \to \Chow(X) 
$$
are given by the addition of cycles. The zero cycle 
$$
\{\varnothing\} = \Chow(X,0) 
$$
is the identity for this operation. For $\cF \in K(\Chow)$ 
such that $\cF\big|_{\Chow(0)}=0$ we define 
$$
\bS_{\Chow} \cF = \bigoplus_{n=0}^\infty \left( (+)_{n,*} \cF^{\boxtimes n}
\right)^{S(n)} \,. 
$$
If we integrate over $\Chow(X)$, this becomes the usual exponential,
that is, 
$$
\chi\!\left(Q^d \, \bS_{\Chow} \cF \right) = \exp\!
\left( \chi\!\left(Q^d \, \cF \right) \right) \,. 
$$

\subsubsection{}

The following is a special case of the main conjecture in \cites{MDT}

\begin{Conjecture}[\cites{MDT}] \label{cM} We have 
$$
\bZ_{K,\PT\to\Chow}  = \bS_{\Chow }\, \bZ_{K,\MM\to\Chow}  \,. 
$$
\end{Conjecture}

\noindent 
This becomes formula \eqref{pptsum} for points in $\Chow(X)$
corresponding to smooth curves.

\subsection{Degree 0 DT counts}\label{s_deg0}

\subsubsection{}

Our discussion so far explicitly ignored DT counts in degree zero like
we did in \eqref{DTPT}. While simpler than curve counts, these
counts played an important role in the development of DT theory as an 
important 
testing ground on which many of the ideas presented above
were developed. 

In particular, Conjecture \ref{cM} and the whole paper
\cites{MDT} were inspired by a conjecture of Nekrasov \cites{NekM} that matches the 
K-theoretic degree 0 Hilbert scheme counts to the contributions of fields of
M-theory to its partition function. 

\subsubsection{}

While degree 0 counts can be defined for an arbitrary $X$, equivariant
localization and the algebraic cobordism ideas of Levine and
Pandharipande \cites{LP} reduce the general case to $GL(3)$-equivariant
computations for $X = \A^3$.  If $\bT \subset GL(3)$ is the maximal
torus, then 
\begin{align*}
  \Hilb(\A^3,n)^{\bT}  &= \{\textup{monomial ideals}\}  \\
&= \{\textup{$3d$ partitions of the number $n$} \}
\end{align*}
and all degree 0 counts are various refinements of the classical count 
\begin{align*}
\Mc(z):=\prod_{n>0} (1-z^n)^{-n} & = \sum_{\textup{$3d$ partitions $\pi$}}
z^{|\pi|} \\
&= \sum_n z^n
                           \chi_\textup{top}(\Hilb(\A^3,n)) \,, 
\end{align*}
that goes back to McMahon. All of these counts may be phrased as sums
over $3d$ partitions that weight $\pi$ by $z^{|\pi|}$ times some
function of the equivariant variables. 

\begin{figure}[!htbp]
  \centering
   \includegraphics[scale=0.75]{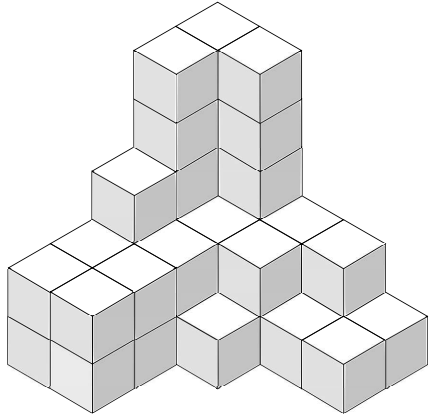}
 \caption{Monomial ideals in $\C[x_1,x_2,x_3]$ of codimension $n$
   correspond to 3-dimensional partitions of the  number $n$.}
\label{f_pile}
\end{figure}

\subsubsection{}
The exact combinatorial nature of these weights will be discussed in
Section \ref{s_0again}, here we only state the results. 

\begin{Theorem}[\cites{MNOP2}]
  \begin{equation}
    \sum_n (-z)^n \int_{\left[\Hilb(\A^3,n)\right]_\vir} 1 =
    \Mc(z)^{\int_{\A^3} (c_3-c_1 c_2)} \,.  \label{HilbC3} 
  \end{equation}
\end{Theorem}

\noindent
Here 
$$
\int_{\A^3} (c_3-c_1 c_2) = - \frac{(t_1+t_2) (t_1+t_3) (t_2+t_3)}{t_1
t_2 t_3}
$$
where $\{t_i\}$ are the Chern roots of the tangent bundle. If $c_1=0$
this reduces to the McMahon identity. 

\subsubsection{}
In the K-theoretic situation, we take $X=\A^3$ in \eqref{localX} with
both $\cL_4$ and $\cL_5$ trivial, so that $Z=\A^5$. On $Z$, we work 
equivariantly with respect to 
\begin{equation}
1 \to \Ct_z \to SL(5)^{z} \to GL(3) \to 1 \label{exactGL3} 
\end{equation}
where $SL(5)$ acts by automorphisms of $Z$ preserving $\Omega^5_Z$. We
denote by 
$$
t_1,t_2, t_3, t_4 = \frac{z}{\sqrt{t_1 t_2 t_3}}, t_5 =
  \frac{1}{z\sqrt{t_1 t_2 t_3}}\,,
$$
the Chern roots of $TZ$, where we picked the most symmetric splitting
of \eqref{exactGL3} for convenience. 

We define the K-theoretic integrand by the same formula as
\eqref{deftODT}
\begin{align}
  \label{deftODT2}
  \tO_{\DT} &= (-1)^{\chi(\cO_C)}
 \cO_\vir \otimes \left( \cK_\vir \otimes
   \det\Hd(\cO_C\otimes(\cL_4-\cL_5)) \right)^{1/2} \notag \\ 
&= (-z)^{\chi(\cO_C)}
 \cO_\vir \otimes \cK_\vir^{1/2}  \,. 
\end{align}
Define 
\begin{equation}
\aroof(t) = t^{1/2} - t^{-1/2} \,. \label{def_aroof}
\end{equation}
The following result was conjectured by Nekrasov \cites{NekM} 

\begin{Theorem}[\cites{PCMI}] We have 
  \begin{align}
    \chi(\Hilb(\A^3,\textup{points}), \tO_{\DT}) &= 
\bSd \left(\chi(Z,T^*Z-T Z)_\textup{moving}\right) \label{sumNK} \\
& = \bSd \, \frac{\prod_{i<j\le 3} \aroof(t_i t_j)}{\prod_{i\le 5}
  \aroof(t_i)} \notag
  \end{align}
where moving in \eqref{sumNK} 
means nonconstant terms in the $z\to 0$ expansion of the corresponding 
rational function. 
\end{Theorem}

\subsubsection{}
The motivation for this remarkable conjecture came from a comparison
with the contributions of \emph{fields} of M-theory to its partition
functions. The fields of M-theory are 
\begin{itemize}
\item[---]  a Riemannian metric on $Z$,
\item[---]  its superpartner 
\emph{gravitino} which is a field of spin $3/2$,
\item[---]  the 3-form field
under which the M2 branes are electrically charged.
\end{itemize}
These are defined 
modulo various gauge equivalences that include diffeomorphisms of
$Z$.  To compute their contributions 
to the partition function is an exercise in representation theory of 
$$
SU(5) \subset SO(10,\R) 
$$
and it gives the remarkable prediction \eqref{sumNK}, modulo a certain
puzzling detail. See Section 3.3 in \cites{PCMI}
for a pedagogical review. 

The puzzling detail is a certain doubling that happens in the answer
\eqref{sumNK}. Namely, fields of M-theory contribute 
$$
\textup{either} \quad \bSd \left(\chi(Z,T^*Z)\right)
\quad 
\textup{or} \quad  
\bSd \left(\chi(Z,-T
  Z)\right)  \,, 
$$
depending on a certain choice. However in \eqref{sumNK} we see the
product of both answers. 

There is an exactly parallel issue in trying to match \eqref{HilbC3} to the GW
counts for $X= \A^3$ which, by the discussion in Section \ref{sHodge}
is the generating function for triple Hodge integrals. Aside from
transcendental constants like $\zeta(3)$, one gets a match between
\eqref{HilbC3}
and the \emph{square} of the GW answer computed in \cites{FP}, see
e.g. \ the discussion in Section 2.4.5 in \cites{OkECM}.  A good explanation
for this doubling phenomenon in degree 0 is yet to be found.

\subsection{Hidden symmetries}

\subsubsection{}
Obviously, there could be more than one $\Ct \subset
\Aut(Z,\Omega^5)$ with 3-dimensional set of fixed points. For
instance, all $\cL_i$ in \eqref{local_curveM} play a completely 
symmetric role. 

The permutation 
$$
\cL_4 \leftrightarrow \cL_5\,, \quad z \leftrightarrow z^{-1} 
$$
is related to the parity of DT counts under $z\mapsto z^{-1}$
discussed in Section \ref{s_parity}. While this is already a deep 
symmetry,  the other permutations are more 
mysterious from the points of view of DT counts. 
In particular, they mix equivariant variables with the boxcounting 
variable $z$.

\subsubsection{}
Permutations like $\sigma=(14)(25)$ 
that preserve the blocks of the partition
$$
N_{X/B} = \cL_1 \oplus \cL_2\,, \quad N_{Z/X} = \cL_4 \oplus \cL_5\,, 
$$
may be interpreted as an instance of the invariance of K-theoretic 
curve counts under \emph{symplectic duality}, also known as the 
3-dimensional mirror symmetry, as well as under other names. 

\subsubsection{}\label{s_fibrHilb} 

{}From both conceptual and technical point of view, it is very
productive to relate DT counts for local curves $X$ as in 
\eqref{local_curve}
to the enumerative theory of sections $f$ of the corresponding 
bundle
\begin{equation}
  \label{fibrHilb}
\xymatrix{
  \Hilb(\A^2,\textup{points}) \ar@{^{(}->}[r]  & \Hilb(\textup{Fiber}) 
\ar[d] \\
& B \ar@/_2pc/_f[u]
} 
\end{equation}
of the Hilbert schemes of points of the fibers in \eqref{local_curve}. In fact, there is a natural
identification 
\begin{equation}
  \label{PT=Quasi}
  \PT(X) = \textup{Quasimap sections $f$ of \eqref{fibrHilb}}
\end{equation}
where quasimaps are maps with certain singularities that will be
discussed presently. 

Very important for this is the fact that 
\begin{align}
  \Hilb(\A^2,n) & = \textup{a Nakajima quiver variety, see \cites{Nakq1,Nakq2}}
  \notag \\
& = \mu^{-1}(0) \rdd GL(n) \label{muGIT} \\
& = \{\textup{Higgs vacua of a certain susy gauge theory}\} \label{HilbHiggs}
\end{align}
where 
$$
\mu: T^* M \to (\Lie GL(n))^* 
$$
is the moment map and $M$ is a certain 
$GL(n)$-module\footnote{For $\Hilb(\A^2,n)$, one takes $M
  = \End(\C^n)\oplus \C^n$, see \cites{NakL}.}. 

The quotient in \eqref{muGIT} is a GIT quotient with a certain choice
of stability conditions, which here concretely means a choice of a 
characters $\chi=(\det)^{\pm 1}$ of $GL(n)$. By definition \cites{CKM}, a quasimap 
$f$ to a GIT quotient is a map 
$$
f: B \to \textup{quotient stack} 
$$
that evaluates to a stable point at all but finitely many smooth points of
$B$. Concretely this means giving: 
\begin{itemize}
\item[---] a $GL(n)$-bundle of prequotients over $B$ with 
\item[---] a section $f$ that generically lands in the stable locus, 
\end{itemize}
where both pieces of data are allowed to vary and are considered
modulo isomorphism. 

GIT quotients of lci
affine algebraic varieties have a technically particularly 
nice enumerative theory
of quasimaps, see \cites{CKM}, and $\Hilb(\A^2,n)$ falls into
this category. 

\subsubsection{}
Enumerative K-theory of quasimaps to quotients of the form
\eqref{muGIT} is a mathematical realization of twisted supersymmetric
indices in certain 3-dimensional gauge theories with 
\begin{align}
  \textup{space-time} &= B \times S^1 \label{BS^1} \\
\textup{gauge group} & = U(n) \notag \\
\textup{matter} & = T^* M  \notag \,. 
\end{align}
One way in which a gauge theory can be in a lowest energy state is
when all gauge fields are constant and all matter fields sit at the
bottom of the potential, giving 
\begin{align*}
\textup{Higgs vacua} &= \textup{minima of the potential} \, \Big/
\, \textup{global gauge} \\
&= \left(\mu^{-1}(c) \cap \mu_\R^{-1}(\theta)\right) / U(n) \\ 
&= \mu^{-1}(c) \rdd_\theta \, GL(n) \,, 
\end{align*}
for some parameters $c,\theta$ in dual of the center of $GL(n)$, where 
$\mu_\R$ is the real moment map. 

A Hamiltonian approach to susy
indices on manifold of the form \eqref{BS^1} is via susy quantum
mechanics, that is, the study Dirac operator on
  the space of 
$$
\textup{modulated vacua} = \big\{\textup{maps $B \to$ Higgs vacua}
\big\}\,,
$$
and this mathematically formalized as an enumerative K-theory of
quasimaps as above. 

\subsubsection{}

In theoretical physics, there is a very powerful set of ideas that
equates such counts for different gauge theories while also mixing
their equivariant and degree-counting variables. Such pairs of gauge
theories are called symplectically dual, 3d mirrors pairs etc., see
e.g.\ \cites{IS,BHOO,BDG,BDGH,BLPW1,BLPW2,NakC1,
NakC2,NakC3}.

In particular, $\Hilb(\A^2,n)$ is self-dual, but the
action of the duality 
on the parameters of the theory is very nontrivial and precisely
correspond to the permutation $\sigma=(14)(25)$ in \eqref{local_curveM}. 

Heuristically, the identification of quasimap counts may be explained
as follows: 
\begin{itemize}
\item[---] moduli of vacua has other irreducible components (known as
  \emph{branches}), and those can be also used to compute susy indices, 
\item[---] the Higgs branch of a gauge theory should be identified with the
  so-called Coulomb branch of the mirror theory, and vice versa. 
\end{itemize}
Recently, there has been a major progress in mathematical understanding
of Coulomb branches, see \cites{NakC1,
NakC2,NakC3}. There is still a very long way to making
the above heuristic rigorous, but there are other ideas and other
 technical tools with
which one \emph{can} prove the equality of quasimap counts, see
\cites{AO, AO2}. We will come back to this in Section \ref{s_actual} below after the
right framework and the right language have been introduced. 

\subsubsection{}\label{s_min_res}
To get more complicated dualities from \eqref{local_curveM} we denote
$$
\Ct_{12}\,, \, \Ct_{45}  \subset \Aut(Z,\Omega^5) 
$$
the groups scaling respective $\cL_i$ with  opposite weights and
consider the subgroups 
$$
\mu_n \subset \Ct_{12}\,, \quad \mu_m \subset \Ct_{45} \,,
$$
of roots of unity of respective orders. The new Calabi-Yau 5-fold 
$$
Z_{n,m} = 
% \begin{matrix}
%   \textup{minimal}\\
% \textup{resolution}
% \end{matrix}
% \,\, 
\textup{minimal resolution}
\left(Z\Big/\mu_n \times \mu_m \right)
$$
fibers over $B$ in $A_{n+1} \times A_{m+1}$, where $A_{n+1}$ is the 
minimal resolutions of the corresponding surface singularity. 
For $Z_{n,m}$ we can take 
$$
\Ct_z = \Ct_{12} \big/ \mu_n  \quad \textup{or} \quad 
\Ct_{45} \big/ \mu_m
$$
and this will correspond to mirror pairs of the form 
\begin{equation}
  \label{Mrn}
  \begin{matrix}
    \textup{moduli of rank $m$}  \\
\textup{bundles
on $A_{n+1}$}  
  \end{matrix} \quad 
\leftrightarrow  \quad 
 \begin{matrix}
    \textup{moduli of rank $n$}  \\
\textup{bundles on $A_{m+1}$}  
  \end{matrix} 
\end{equation}
where the moduli spaces\footnote{
technically, moduli of framed torsion-free sheaves of a certain rank
like we saw in Section \ref{s_torsion_free}.} are
higher rank brothers of the Hilbert schemes of points in the
respective surface and, like their rank 1 siblings, they are
examples of Nakajima quiver varieties, see \cites{NakL}. 

In this example, the fixed locus $Z_{n,m}^{\Ct_z}$ has $n$ or $m$
components. The interactions between these components, see \cites{MDT}, 
which we did not discuss here, are
similar to the interaction between partitions in Section \ref{s_rtuple} and
they similarly correspond to having more than one section in
\eqref{cOq}. 

\subsubsection{}\label{s_engin}

Curve counts in \eqref{Mrn} contain, in particular, classical
computations in equivariant K-theory of the corresponding varieties,
which is the subject of Nekrasov theory. This is how Nekrasov theory
can be \emph{engineered} from the DT theory, that is, this is how 
Nekrasov counts can be seen as instances of counting 
M2-branes or other curves. 

For the
record, in theoretical physics, it has been understood long ago
\cites{Katz_eng}
that M-theory reduces to
corresponding supersymmetric gauge theories in the case at hand, and
in particular some form of this connection was clear to Nekrasov at
the time the theory of \cites{NekInst} was created, see in particular 
Section 4 in \cites{NekInst}. Still, it is nice to see a precise
match appear as a special case of general mathematical conjectures.

\section{Toric DT counts}\label{s_toric}

\subsection{Degree 0 again}\label{s_0again}

\subsubsection{} 

For ease of writing formulas like \eqref{bEh}, it is convenient 
to extend \eqref{def_aroof} to a map 
$$
\aroof:  K_\bT(\pt)=\Z T^\wedge \to \Q\left(\sqrt{\bT}\right) 
$$
by the rule 
$$
\aroof\left( \sum m_i w_i \right) = \prod 
\left(w_i^{1/2} - w_i^{-1/2}\right)^{m_i}\,,
\quad m_i \in \Z, \,\, w_i \in \bT^\wedge \,. 
$$
Here $\bT^\wedge$ are the characters of $\bT$ and 
$\sqrt{\bT}$ is the torus with characters $w^{1/2}$, for all 
$w\in \bT^\wedge$.  This is something we have seen this before, for if $V$ is a $\bT$-module then 
$$
\aroof( - \textup{character of $V$}) = 
\textup{character of} \,  \, \bSdh V \,,
$$
where $\bSdh V$ is the symmetrized symmetric algebra from 
\eqref{bSdhsymm}.

\subsubsection{} 

The discussion of Section \ref{s_NT} generalizes verbatim for
$\Hilb(\A^3)$ in place of $\Hilb(\A^2)$ if we replace the structure
sheaves etc.\ by their virtual analogs \eqref{deftODT2}. 
Equivariant localization works for virtual counts \cites{GP,FGoe} and gives 
\begin{align}
 \bZ_{\DT}(\A^3) &=  \chi(\Hilb(\A^3,\textup{points}), \tO_{\DT} )
\notag \\
& = \sum_{\textup{$3d$ partitions $\pi$}} (-z)^{|\pi|} 
\,\, \aroof\left(-T^\vir_{I_\pi} \Hilb \right) \,, \label{ZA3}
\end{align}
where the virtual tangent space is the difference
$$
T^\vir_{I_\pi} \Hilb = \Def_{I_\pi} - \Obs_{I_\pi}
$$
between the 
deformations 
and obstruction spaces at the monomial ideal $I_\pi$. These
correspond to $\Ext^{i}(I,I)$, with $i=1,2$ respectively, and give the 
old\footnote{In fact, interpreted to include obstructions to
  obstructions, etc.\ this formula hold in any dimension. It is
  difficult, though, to incorporate 
 such higher obstructions into enumerative
  counts, which is why $\dim X = 3$ is special.} 
formula \eqref{ext1chichi} 
\begin{equation}
T^\vir_{I_\pi} = \chi(\cO) - \chi(I_\pi,I_\pi) \,. \label{Tvirpi}
\end{equation}

\subsubsection{} 

In parallel with \eqref{bV2}, we define 
\begin{equation}
\bG_\pi = \chi_{\A^3} (\cO/I_\pi) = \sum_{\bx = (i,j,k) \in \pi} 
t_1^{1-i} t_2^{1-j} t_3^{1-k} \,, \label{bV3}
\end{equation}
where the boxes correspond to monomials as follows 
$$
\pi \owns \bx = (i,j,k) \,\, \leftrightarrow \,\, x_1^{i-1} x_2^{j-1} 
x_3^{k-1} \notin I_\pi \,. 
$$

\begin{Lemma}
The character of \eqref{Tvirpi} is given by 
\begin{equation}
  \label{Tvirf}
 T^\vir_{I_\pi} = \bG_\pi  - t_1 t_2 t_3\, 
\overline{\bG}_\pi - 
(1-t_1)(1-t_2)(1-t_3) \, \bG_\pi \overline{\bG}_\pi \,. 
\end{equation}
\end{Lemma}

\noindent
The proof is the same as the proof \eqref{Ext1f1} in Lemma
  \ref{lemt1}. It would be nice to have a good analog of 
\eqref{Ext1f2}, but one should bear in mind that, unlike 
\eqref{Ext1f2}, \eqref{Tvirf} has an equal number of positive and
negative monomials, and this number can be very large. 

\subsubsection{}
With this practical description of the left-hand side in 
\eqref{sumNK}, one can use e.g.\ a
computer algebra system to expand \eqref{sumNK} in powers of $z$ 
and get a sense of the nontriviality of this identity. In general,
boxcounting computations are a great resource for both practitioners
and students of the DT theory. It is very rewarding to see abstract
theories agree with computer experiments, and one is very often guided
by the latter in the search for correct general formulations. 

\subsection{Torus-fixed subschemes}

\subsubsection{}
Now suppose $X$ is a smooth quasiprojective toric variety, where toric means
that a $3$-dimensional torus $\bT$ acts on $X$ with an open orbit, or that
$X$ is glued from $\A^3$-charts using monomial transition
functions. As an example one can take $X= \bP^3, (\bP^1)^3$, or 
a local $\bP^1$ in \eqref{localX}. 

The combinatorics of $X$ is best captured by the corresponding
polyhedron $\Delta(X)$ ---  the image of the
real moment map in $(\Lie \bT)^*$, see Figure \ref{f_D(X)}. The
$k$-dimensional faces of $\Delta(X)$ are in bijection with 
$k$-dimensional torus orbits and, in particular, reduced irreducible 
$\bT$-invariant curves $\bL_\be\subset X$ correspond to the edges $\be \in \Delta(X)$. 
For example, in $\bP^3$ there are 6 such curves --- 3 coordinate lines
in $\A^3$ and 3 coordinate lines in the projective plane at infinity. 
\begin{figure}[!htbp]
  \centering
   \includegraphics[scale=0.64]{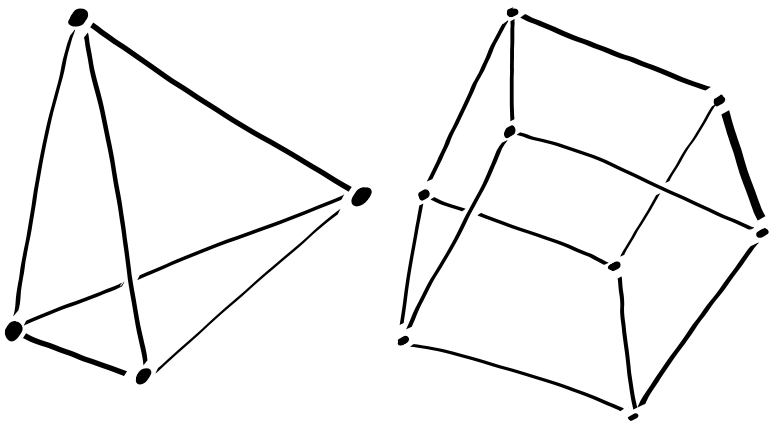}
 \caption{We have $\Delta(\bP^3)=$ tetrahedron, while 
$\Delta((\bP^1)^3)=$ cube.}
\label{f_D(X)}
\end{figure}

\subsubsection{}

The vertices $\bv\in \Delta(X)$ correspond to toric charts $\A^3_\bv
\subset X$. In such a chart, a $\bT$-invariant 
subscheme $C \subset X$ has to look like the ideal in Figure \ref{f_Ipi3}, 
that is, like a 3-dimensional partition $\pi_\bv$ which can have infinite legs in the 
coordinate directions. We have already 
seen a 2-dimensional version of this in Figure \ref{f_part2}. The infinite 
legs in Figure \ref{f_Ipi3} end asymptotically on three 2-dimensional
partitions $\lambda_{\vbe}$ indexed by the 
oriented edges 
$\vbe$ emanating from the vertex $\bv$. 

\begin{figure}[!htbp]
  \centering
   \includegraphics[scale=0.7]{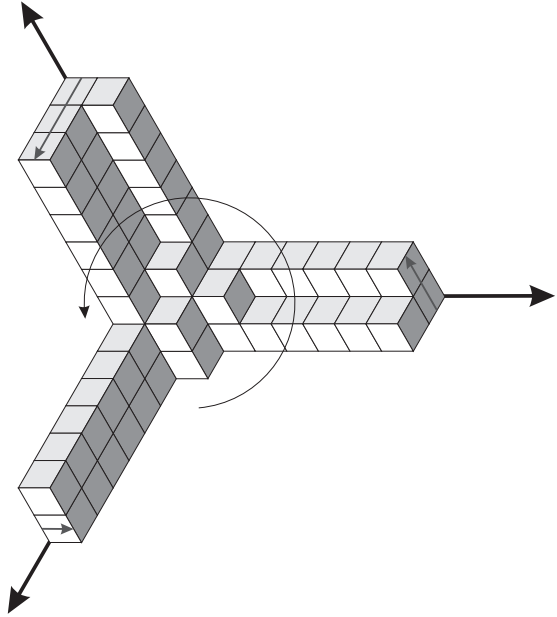}
 \caption{A one-dimensional 
ideal $I_\pi \in \C[x_1,x_2,x_3]$ corresponds to a 
3-dimensional partitions that can have infinite legs in 3-coordinate
directions. Just to label their asymptotic $2d$ partitions
$\lambda_\be$, it is convenient to use the
orientation of the figure that comes from the orientation of the
$\Delta(X)$. In this example, $\lambda_\be\in \{(2,1),(3,1),(1,1)\}$.}
  \label{f_Ipi3}
\end{figure}

Globally these partitions glue like in Figure \ref{f_edge}, where the length
of the edge should be considered as something much, much larger than
the size of a box. For the 3-dimensional partitions to glue, their
asymptotic partitions $\lambda_\be$ along the given edge have to match
like 
$$
\lambda_{\overrightarrow{\be}} =
\left(\lambda_{\overleftarrow{\be}}\right)' 
$$
if one follows the conventions of Figure \ref{f_Ipi3} and where prime
denotes
the transposed partition.

\begin{figure}[!htbp]
  \centering
   \includegraphics[scale=0.7]{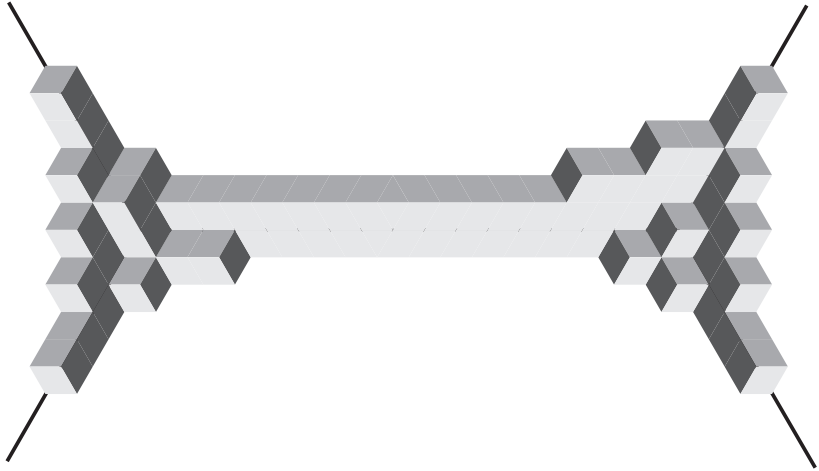}
 \caption{Two one-dimensional ideals glue on the overlap 
of two toric charts. The edge scheme $C_\be$ here is the line 
$\bL_\be$ doubled in one of the coordinate directions.}
\label{f_edge}
\end{figure}

\subsubsection{}

If we take just this partition $\lambda_\be$ running along $\bL_\be$,
we get a canonically defined subscheme 
$$
C \supset C_\be = 
\overline{C \cap \A^3_\bv \cap \A^3_{\bv'}} 
$$
where $\bv,\bv'$ are the two vertices joined by the edge $\be$. We
have
$$
\deg C_\be = |\lambda_\be | \deg \bL_e
$$
and 
$$
Q^{\deg C} = \prod_{\be \in \Delta(X)}  (Q^{\deg
  \bL_\be})^{|\lambda_\be|}
$$
is the $Q^d$ weight\footnote{Note that since the weight of one 
box is $(-z)$, one can interpret $\ln Q^{\deg \bL_e}/\ln(-z)$ as the
length of $\be$.}
of $C$ in partition functions \eqref{ZXD2} and 
\eqref{ZXDK}.

\subsection{Localization of DT counts}

\subsubsection{}
With the kinematics of $\bT$-fixed subschemes sorted out, we now
discuss the dynamics, that is, their weights in localization
formulas. Just like in \eqref{Tvirpi} we have 
$$
T^\vir_C  = \cT_X (\cI_C)
$$
where $\cI_C$ is the ideal sheaf of $C$ and where we denote 
$$
\cT_U(\cF) = \chi(\cO_U) - \chi(\cF,\cF)\,, 
$$
for a coherent sheaf  $\cF$  on an open set $U\subset X$.
 It is convenient to
have $\cT_U(\cF)$ defined on opens because e.g.\ 
 localization gives
\begin{equation} 
T^\vir_C = \sum_{\bv \in \Delta(X)} \cT_{\A^3_\bv}(I_{\pi_\bv}) \,.  \label{TvirLoc}
\end{equation}

Each term in \eqref{TvirLoc} may be computed by the formula 
\eqref{Tvirf}, where $\bG_{\pi_\bv}$ is now a rational function of 
$(t_1,t_2,t_3)$. Similarly, $\cT_{\A^3_\bv}(I_{\pi_\bv})$ is 
a rational function, and so the computation \eqref{TvirLoc} is really
taking place in the localized $\bT$-equivariant K-theory. The torus
$\bT$ has no fixed points on double intersections of the charts 
$\A^3_\bv$, which is why the corresponding terms are absent in
the sum\eqref{TvirLoc}. Their $\bT$-characters are torsion and vanish in the
localization. 

On the other hand, being a rational function makes 
$\cT_{\A^3_\bv}(I_{\pi_\bv})$ unsuitable as an argument 
for $\aroof(\, \cdot \,)$. To remedy this, we literally subtract the contributions of the
infinite legs as follows
\begin{equation}
\cT^\vtx_{\pi_\bv} = \cT_{\A^3_\bv}(I_{\pi_\bv})  - 
\sum_{\be} \cT_{\A^3_\bv}(I_{\be}) \label{defcTvtx}
\end{equation}
where the sum is over the edges $\be$ incident to $\bv$ and 
$I_\be= \cI_{C_\be}\big|_{\A^3_\bv}$. It is easy to see this 
\eqref{defcTvtx} is a 
\emph{polynomial} in $t$. 

By construction, this splits the virtual 
tangent space 
\begin{equation} 
T^\vir_C = 
\sum_{\be \in \Delta(X)} T^\vir_{C_\be} + 
\sum_{\bv \in \Delta(X)} \cT^\vtx_{\pi_\bv} 
\,. 
\label{Tvirvtx}
\end{equation}
into the contributions of the edges curves $C_\be$ and their
interaction at the vertices $\bv$.  This can be used to 
organize the localization formula as follows. 

\subsubsection{}
If $\lambda$ is a partition running along an edge $\be$, we define the
corresponding edge weight by 
$$
\bE_\be(\lambda) = (-z)^{\chi(\cO_{C_\be})} \, 
Q^{|\lambda| \deg L_\be} \,\,  \aroof\left(-  T^\vir_{C_\be}\right) \,. 
$$
There is an elementary arm-and-leg expression for 
$T^\vir_{C_\be}$ that depends on the normal bundle 
$N_{X/\bL_e}$. In
particular, $\bE_\be(\lambda)$ 
becomes the function $\bEN(\lambda,\lambda)$ from 
Section \ref{s_NT} when the normal bundle is
trivial. 

If $\lambda$, $\mu$, and $\nu$ are the partitions running along the
three edges incident to a vertex $\bv$ we define 
\begin{equation}
\bV_\bv(\lambda,\mu,\nu) = \sum_{\textup{$\pi$ ending on $(\lambda,\mu,\nu)$}}
(-z)^{|\pi|} \,\,\,  \aroof\left(- \cT^\vtx_\pi \right) \label{defVert}
\end{equation}
where $|\pi|$ is the regularized size of $\pi$, namely
$$
|\pi| = \chi(C_\pi)- \sum_\be \chi(C_\be) \,, 
$$
which may be negative. 

Note that \eqref{defVert} depends on the vertex $\bv$
only through the assignment of equivariant variables. In other words,
this is one universal function of 3 partition, 3 equivariant
variables, and the box-counting variable $z$. It has an
$S(3)$-symmetry that permutes/transposes partitions while
permuting the equivariant variables. 

\subsubsection{}

With this notation, we can write the localization formula 
for curve counts in $X$ as a partition function of a vertex model, in
which the degrees of freedom are partitions living on the edges of
$\Delta(X)$ and their interaction happens at vertices. 
Concretely,
\begin{equation}
  \label{ZDvert}
Z(X) = \sum_{\textup{all maps $\be\to \lambda_e$}}
\prod_\be \bE_\be (\lambda_\be) 
\prod_\bv \bV_\bv (\lambda_{\be_1}, \lambda_{\be_2},
\lambda_{\be_3}) 
\end{equation}
where $\be_i$ are the three edges incident to $\bv$. 

\subsubsection{}\label{s_top_vertex}

Formulas of the form \eqref{ZDvert} have a long history in the
subject, starting from the \emph{topological vertex} conjecture of
\cites{AKMV}. In there, Aganagic, Klemm, Mari\~no, and Vafa proposed a
formula, with the same structure, for GW counts in toric CY 
threefolds. The topological vertex of \cites{AKMV} is an explicit
expression with Schur functions inspired by a connection to knot invariants
from Section \ref{s_CSknots}. 

Note that in the Calabi-Yau case, we have
$$
T^\vir = - \left(T^\vir \right)^\vee
$$
equivariantly. Hence $\aroof(T^\vir) = \pm 1$ and the summation in \eqref{ZDvert} becomes pure
combinatorics of boxcounting. From a modern point of view, this
combinatorics had just been sorted out at the time 
in \cites{OR}, and so once one 
knew\footnote{What was noticed first was the equality of the limit
  shape of \cites{OR,CerfK} for 3d partions with the GW-mirror of $\C^3$. One
  can, in fact, get quite far by interpreting mirrors as limit
  shapes, see e.g.\ \cites{OkECM,OkAMS}.} there is a connection 
it was easy to see that 
$$
\textup{topological vertex} = \sum_{\textup{$\pi$ ending on $(\lambda,\mu,\nu)$}}
z^{|\pi|}\,, \quad z = e^{iu} \,. 
$$
This is the main point of \cites{ORV} and it 
was also explained there that boxcounting is related to the Hilbert
scheme of curves in 3-dimensions. An early discussion of a
3-dimensional analog of Nekrasov theory may be found in a related
paper \cites{Iqbal}. 

All of this was a very important inspiration for \cites{MNOP1,MNOP2}, where the
general GW/DT correspondence was proposed and where it was explained
how it specializes to the topological vertex formula for
toric CY threefolds.

\subsubsection{}

Formulas like \eqref{ZDvert} have an obvious gauge symmetry.
They can be seen as a contraction of certain tensors in 
$$
\Fock^{\otimes \textup{edges of $\Delta(X)$}}(t_1,t_2,t_3)
(\!(z)\!)
$$
where $\Fock$ is a vector space with a basis given by
partitions. Clearly, we are free to change the basis in any tensor
factor without affecting the result.

In \eqref{ZDvert}, the edge terms are explicit products and the whole
complexity sits at vertices. The \emph{capped} localization of
\cites{MOOP}, is a gauge transformation that spreads the complexity
more evenly. In capped 
localization, edge terms corresponds to deformations of $C_\be$ 
considered \emph{relative} to the two toric divisors at the endpoints of $\bL_\be$,
and similarly, the vertex terms are considered relative infinity
which may be modeled by the infinity of 
$$
\A^3 \subset (\bP^1)^3  \,. 
$$
The equivalence with \eqref{ZDvert} follows from the degeneration
formula of Section \ref{s_gluingDT}. 

\subsubsection{}

In capped localization, the edge terms are well-controlled rational
functions of $z$ that absorb some of the complexity from 
\begin{align}
  \label{capped_vertex1}
  \frac{\bV(\lambda,\mu,\nu)_\textup{capped}}{\bV(\varnothing, \varnothing,
  \varnothing)} &= \textup{a rational function of $z$, in fact} \\
& \overset{?}= \textup{a Laurent polynomial in $z$} \,. 
\label{capped_vertex2}
\end{align}
The rationality \eqref{capped_vertex1}, which is the analog of the
Conjecture \eqref{cr}, 
is proven in \cite{MOOP} along the lines that will be explained
shortly. In fact, 
one of our main 
goal in the rest of these notes is to
explain how one computes this function. 

The polynomiality in 
\eqref{capped_vertex1} remains a conjecture.

\subsubsection{}
Localization in PT theory takes a very similar shape, see \cites{PT2}. 
The fixed loci now have much fewer components, but may be not isolated
if $\lambda,\mu,\nu$ are all nonempty.

\subsubsection{}

The shape of localization formulas in GW theory is structurally
very similar. Let 
$$
f: C = \bigcup C_i \to X
$$
be a $\bT$-fixed stable map, where $\{C_i\}$ are irreducible
components of the source curve $C$. For each $C_i$, there are two
possibilities: 
\begin{itemize}
\item[---] if $C_i$ is not contracted by $f$, then $f|_{C_i}$ has the
  form
$$
C_i \cong \bP^1 \owns y \xrightarrow{\,\, f\,\,} y^{\mu_i} 
\in \bP^1 \cong \bL_\be
$$
for some edge $\be$ and some degree $\mu_i=1,2,\dots$. 
For every edge, these degrees form a partition $\mu_\be$. The
contribution of such $C_i$ to 
 the virtual tangent space is simple and explicit. 
\item[---] otherwise, $C_i$ is contracted by $f$ to a vertex
$\bv$. These components contribute an analog of \eqref{psilam}, but now with 3
Chern classes of the Hodge bundle. 
\end{itemize}

There is, similarly, a capped version of the localization available. The proof of the GW/DT correspondence for toric varieties \cites{MOOP}
really matches the capped GW vertices and edges to their cohomological
counterparts in DT theory. 

\begin{figure}[!htbp]
  \centering
   \includegraphics[scale=0.85]{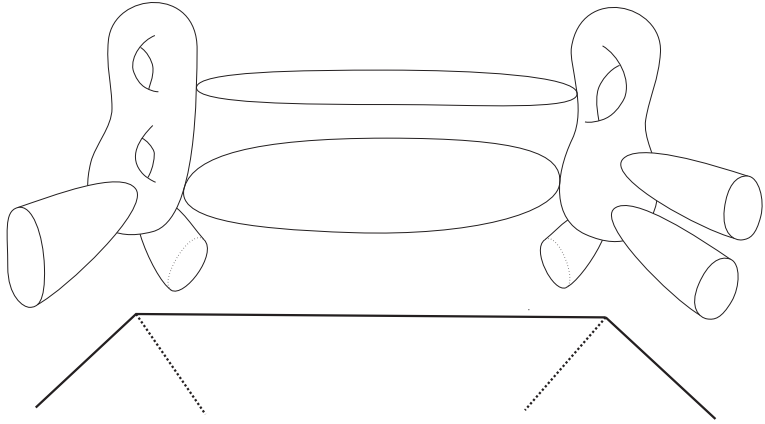}
 \caption{The shape of torus fixed stable maps to $X$, compare with
   Figures \ref{f_Hurw} and \ref{f_Ipi3}.}
\label{f_locGW}
\end{figure}

\subsection{$A_n$-geometries capture vertices}
\label{s_Ancap} 

\subsubsection{}
The vertex \eqref{defVert} is a basic building block of the theory and
a very important special function. One can view it as a tensor 
\begin{equation}
\bV \in (\Fock)^{\otimes 3} \otimes \Q(t) (\!( z )\!) \label{bVFock3}
\end{equation}
and contemplate determining it by the representation theory of a
suitable algebra acting in this linear space, like it was done in 
\cites{ADKMV,AFS} for the topological vertex, and the so called
refined topological vertex, which is also a (very) special case of
\eqref{defVert}. 

While this is a very important direction of current
research, it turns out to be easier to repackage this tensor
differently, as a certain different 3-valent tensor that can be
determined by representation theory of a certain quantum 
group. This group is $\cU_\hbar(\glhh_3)$, which is a quantum loop group
associated to the loop algebra $\glh_3$, see \cite{PCMI} for an introduction. There is, in fact, a sequence
of $n$-valent tensors that similarly correspond to $\glh_n$ for all
$n$. They correspond to 
\begin{equation}
X_n=A_{n-1} \times \A^1\,, \quad n=1,2,\dots\,,\label{eq:1}
\end{equation}
where $A_{n-1}$ is 
the toric symplectic surface we met previously in Section
\ref{s_min_res}.

The vertex  \eqref{bVFock3} has a nice $S(3)$-covariance, but 
the fact that its three legs point in three different coordinate
directions makes it more difficult for the three tensor factors in 
\eqref{bVFock3} to interact. 

\subsubsection{}
We have 
$$
A_1 = T^* \bP^1
$$
and the toric polyhedron $\Delta(X_2)$ is drawn in Figure \ref{f_A2}. 
\begin{figure}[!htbp]
  \centering
   \includegraphics[scale=0.6]{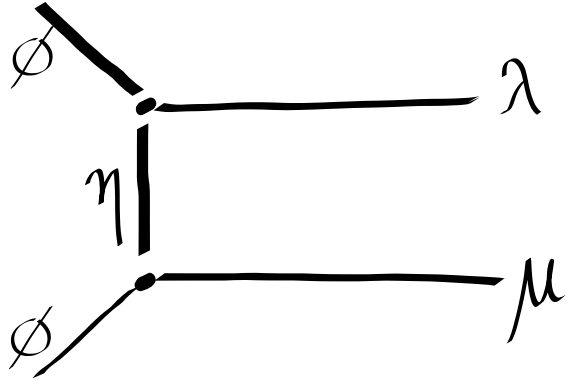}
 \caption{The $1$-skeleton of $\Delta(X_2)$ has a unique bounded 
edge and the corresponding partition function \eqref{ZX2} is a sum over
the unique intermediate partition $\eta$.}
\label{f_A2}
\end{figure}
We assign the partitions $\varnothing$, $\lambda$, $\mu$, and 
$\varnothing$ to the unbounded edges and consider
\begin{equation}
  \label{ZX2}
  \bZ(X_2,\lambda,\mu\, | \, z,Q) = \sum_\eta \bE(\eta) \, \bV(\varnothing,\mu,
\eta') \, \bV(\eta,\lambda,
\varnothing) \,. 
\end{equation}
where the new variable $Q$ enters through the  $Q^{|\eta|}$ 
part of the edge weight $\bE(\eta)$ in \eqref{ZX2}. The coefficient 
$Q^k$ thus restricts the summation to $|\eta|=k$. 

The sum \eqref{ZX2} may be seen as the partition function 
for $A_1 \times \bP^1$ with nonsingular boundary conditions imposed at
the divisor $A_1 \times \{\infty\}$. 

\begin{Lemma}[\cite{MOOP}] 
The function \eqref{ZX2} uniquely determines the vertices
with one of the 3 partitions empty. 
\end{Lemma}

\begin{proof}
Suppose we want to compute $\bV(\eta,\lambda,
\varnothing)$. By symmetry and induction, we can assume $|\eta| \le
|\lambda|$. Viewing \eqref{ZX2} as linear equations on the unknowns 
$\{\bV(\eta,\lambda,
\varnothing)\}$, it suffices to check that the matrix
$$
\left(\bV(\varnothing,\mu,
\eta')\right)_{|\mu|<k\,, |\eta|=k}
$$
has maximal rank for $k=1,2,\dots$. It is enough to check this for the
topological vertex specialization, and this was done in \cite{MOOP}. 
\end{proof}

\subsubsection{}
It is now clear that the same argument applied to 
$$
X_3 = A_2 \times \A^1
$$
will capture the full 3-valent vertex with one of the outgoing edges
equal to $\mu$. 
\begin{figure}[!htbp]
  \centering
   \includegraphics[scale=0.6]{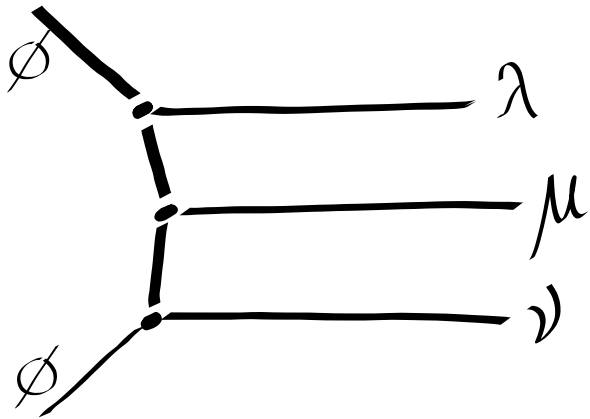}
 \caption{The topic polytope for $X_3$.}
\label{f_A3}
\end{figure}
As a result, the function $\bZ(X_3,\lambda,\mu,\nu)$ contains 
$\bV$ in an effective way, from which properties like 
rationality in $z$ may be concluded. 

Clearly, $\bZ(X_3,\lambda,\mu,\nu)$ a richer function since it
depends on $Q_1$ and $Q_2$ for the two curve classes in $A_2$. 
It is therefore remarkable that this function, and its $n$-valent 
generalization for $X_n$ may be described explicitly in the language 
of geometric representation theory. 

\section{DT counts in ADE fibrations} \label{s_actual} 

\subsection{Quasimaps again} 

\subsubsection{}
The ADE surfaces $S$, minimal resolutions of 
$$
\textup{ADE singularities} = \A^2 / \Gamma \,, 
$$
where $\Gamma \subset SL_2(\C)$ is a finite subgroup, are some of the
nicest objects in the theory of surfaces and very often make an
appearance in other parts of mathematics. 

What unites them from our
point of view is that their Hilbert schemes are Nakajima varieties
\cite{Nakq1}. Therefore if $X$ is an ADE fibration
\begin{equation}
  \label{fibrS}
\xymatrix{
  S \ar@{^{(}->}[r]  & X
\ar[d]^\pi \\
& B 
} \,, 
\end{equation}
like what we have already seen in Section \ref{s_min_res}, then 
one can study quasimap sections of the corresponding 
fibration\footnote{Normally, what we denote by 
$\Hilb(\pi)$ is called the relative Hilbert scheme of $\pi$, but the
meaning of the world \textsl{relative} has been reassigned in DT theory.} of 
the Hilbert schemes 
\begin{equation}
  \label{fibrHilbS}
\xymatrix{
  \Hilb(S) \ar@{^{(}->}[r]  & \Hilb(\pi) 
\ar[d] \\
& B \ar@/_2pc/_f[u]
} \,, 
\end{equation}
like we did in Section \ref{s_fibrHilb}. 

\subsubsection{}

If $S=\A^2$ then these quasimap moduli spaces are the PT moduli
spaces. In general, they are even more economical than the PT moduli
spaces and, by transitivity, much more economical than the Hilbert
schemes of curves in $X$. 

Quasimap counting 
 also has a noticeable technical 
advantage in how it can handle sheaves of higher rank on $X$. Recall from 
Section \ref{s_min_res} 
that the moduli of rank $r$ torsion free sheaves on $S$ are also 
Nakajima varieties, and so one can just count quasimaps to those. 

Counting quasimaps in \eqref{fibrHilbS} is just another flavor of DT
theory of X and is expected to be related to other counts by a simple 
wall-crossing like \eqref{DTPT}. In cohomology, this comes out 
for free from the results of \cites{MOOP,MObl}, and situation in 
equivariant K-theory should be similar. 

\subsubsection{}

Since it is easier to work with quasimaps, and K-theoretic counts are
complicated enough, so far people focused on doing them for
quasimaps. There is now a satisfactory general enumerative theory of
quasimaps to general Nakajima varieties, see \cite{PCMI,slc} for an
introduction and overview. 

\subsubsection{}

The base $B$ in \eqref{fibrS} is a fixed nonsingular curve. By
deformation invariance and gluing, the DT counts relative the
fibers of \eqref{fibrS}, that is, the assignment 
\begin{equation}
(B,\{b_i\})  \mapsto \bZ(X/D)\,, \quad D=\pi^{-1}(\{b_i\})\label{Bbi}
\end{equation}
defines a  2-dimensional 
TQFT and, more generally, a K-theoretic analog of a cohomological
field theory if the curve $B$ is allowed to move in families. 

The space of boundary 
conditions in this TQFT is 
\eqref{Fock} in cohomology and 
\begin{equation}
  \label{FockK} 
 K_{\Aut(S)}(\Hilb(S)) = \bigoplus_k K_{\Aut(S)}(\Hilb(S),k) 
\end{equation} 
in equivariant $K$-theory. 

\subsubsection{}

More precisely, this is a TQFT that is further enriched to keep
track of the topology of the bundle in \eqref{fibrS} which is 
labelled by a cocharacter
\begin{equation}
\sigma: \Ct \to \Aut(S)  \,. \label{sigma_cochar}
\end{equation}
In degeneration, the bundle degree may be split arbitrarily between 
the components.  

It is convenient to treat bundle topology
as the action of commuting \emph{shift operators}  
indexed by $\sigma$. In higher rank, there are similar shift
operators related to the framing automorphisms of Nakajima 
varieties.

\subsubsection{}

There is a simple Lie algebra $\fg_\ADE$ associated to $S$ and 
the corresponding affine Lie algebra
\begin{equation}
\fg_\KM  = \textup{central extension of  } \fg_\ADE[x^{\pm
  1}]\label{g_KM}
\end{equation}
is an example of a Kac-Moody Lie algebra. 
By the work of Nakajima \cite{Nak3}, \eqref{FockK} is a module 
for $\cU_\hbar(\gh_\KM)$, which is a quantum loop group
algebra associated to the Kac-Moody 
Lie algebra \eqref{g_KM}.  

Note that we are dealing here with quantum \emph{double} loop
symmetries, something that has many more layers of
representation-theoretic complexity than the affine Lie algebra or 
the ordinary (non-loop) quantum groups of the CS theory. 

\subsubsection{}

{} From a geometric R-matrix perspective, an abstract Lie 
algebra $\fg_\MO$ was associated in \cite{MO} to any quiver so that 
$\cU_\hbar(\gh_\MO)$ acts on equivariant K-theories of 
Nakajima quiver varieties, see \cite{OS,PCMI}. 
This extends the action of $\cU_\hbar(\gh_\KM)$ constructed by
Nakajima in \cite{Nak3}. 

For quivers of affine ADE type, the difference between the two Lie
algebras is a single Heisenberg algebra $\glh_1$. For example 
$$
\fg_\KM  = \slh_{n+1}  \subset \glh_{n+1} = \fg_\MO
$$
for quivers of affine $A_n$-type that correspond to 
Hilbert schemes of $A_n$-surfaces.  It is the action of 
$\cU_\hbar(\gh_\MO)$ that will be required to do the quasimap
counts and the main motivation for \cite{MO} were precisely
enumerative applications. 

\subsection{Representation theory answers}

\subsubsection{}

{} From the TQFT principles, it suffices to do the counts \eqref{Bbi} 
for $B=\bP^1$ with at most 3 marked points $\{b_1,\dots,b_n\}$. 
The counts with $n\le 2$ marked points can be done 
equivariantly with respect to $\Aut(B,\{b_i\})$. In fact, this makes
$\le 2$-point counts so rich that they 
determine all other ones by elementary operations. 

Recall from Section \ref{s_boundary} that we have several options for
the boundary conditions to impose at $\pi^{-1}(\{b_i\})$. 
Namely, we can impose the nonsingular, relative, and descendent 
boundary conditions, see Section \ref{s_flavors}. 
 Counts with 2-marked points and different flavors
of boundary conditions provide the dictionary between the insertions
as in  Section \ref{s_corresp}.  

Different flavors of insertions results in objects of different
geometric, analytic (as functions of parameters), and
representation-theoretic flavors. We now list 
various possibilities, first taking the $S$-bundle in 
\eqref{fibrS} to be trivial. 

\subsubsection{One nonsingular point}\label{s_1ns}

This is the count of Section \ref{s_Ancap}, le raison d'etre of this
section. By change of the boundary condition, it corresponds to: 

\subsubsection{One relative point}

This count is trivial, in the sense that this is identity element (structure
sheaf $\cO_{\Hilb}$) in \eqref{FockK}, up to 
normalizations. 

While 
the count is trivial, this triviality is in itself highly 
  nontrivial as it involves the vanishing of many quantum corrections
  that are not constrained to vanish for simple reasons like
  dimension. Here, we have an instance of what was called 
\emph{large rank vanishing} in \cite{PCMI}, something which also holds
for counts of Section \ref{s_rel_desc} once the rank becomes
large compared to the size of the descendent insertion. 

What relates this simple answer to Section \ref{s_1ns} is the
following change of basis
matrix: 

\subsubsection{Two points: one relative, one nonsingular}
\label{1rel1ns}

This is a fundamental solutions of a certain system of 
$q$-difference equations in all other variables. Here 
$$
q \in \Ct = \Aut(B, \{b_1,b_2\}) 
$$
and the equations are written using the operators of 
$$
\cU_\hbar = \cU_\hbar(\gh_\MO)\,,
$$
 see \cite{OS, PCMI, slc}. 

%  In other words, 
% this is an infinite product of certain elements of 
% $\cU_\hbar$. 

The equations in the variables $z$ and $Q$ found in 
\cite{OS} may be 
described as a certain deformation of the action of what may be 
called the affine Weyl group of $\cU_\hbar$. 
Beware, however, that in the non-Kac-Moody situation we are
really talking not about a group but about the fundamental groupoid of
a certain hyperplane 
arrangement, namely the arrangement associated to the roots of
$\gh_\MO$.  This arrangement is very far from being the 
arrangement of reflection hyperplanes for a discrete reflection
group. In particular, its
symmetries, which apriori include only a lattice of translations and
the reflection about
the origin,  do not act transitively on the 
set of alcoves.  

For instance, for $\Hilb(\A^2,k)$, this is 
a $1$-dimensional arrangement with hyperplanes
$m x \in \Z$ for $m=1,\dots,k$, where $x$ is the coordinate. 
 This kind of combinatorics
replaces Weyl group symmetries in the subject, see \cite{slc} for a
discussion. 

A fundamental solution to a $q$-difference 
equations may be interpreted as a (regularized) infinite product of 
certain wall-crossing operators that cross all walls between a certain 
alcove and a point at infinity.

\subsubsection{Two points: one relative, one descendent}
\label{s_rel_desc}

There are explicit answers for these counts, too, see
\cites{AObethe,Sm1}. By general principles, they give a dictionary 
between relative and descendent insertions, something which is very
useful for accessing DT counts in general threefolds.  

They also 
related the subject of Bethe Ansatz for $\cU_\hbar$, 
because they relate the 
counts of Section \ref{1rel1ns} to the following: 

\subsubsection{Two points: one descendent, one nonsingular}
\label{1desc1ns}

These counts may be given by an explicit Mellin-Barnes-type intergral
because the $q$-fixed quasimaps are GIT quotients, see \cite{afo}. 
Integral solutions for $q$-difference equations turn into formulas
for eigenfunctions in the $q\to 1$ limit and this is the subject 
of Bethe Ansatz for $\cU_\hbar$. Geometric considerations solve the
problem completely, giving, in particular,
 a concise explicit formula for general off-shell Bethe
 eigenfunction, see \cites{AObethe}. 

Of course, the corresponding \emph{eigenvalues} were understood long
time ago by Nekrasov and Shatashvili in \cites{NS1,NS2}. That very
influential work was what stimulated the interest in quasimap counts in the
first place, see e.g.\ \cites{MO} for a longer discussion. 

\subsubsection{Two point, both relative, nontrivial bundle}

This is the shift operator corresponding to the degree $\sigma$ of the
bundle in \eqref{sigma_cochar}. They give the $q$-difference equations 
from Section \ref{1rel1ns} in the corresponding equivariant
variables. 

These operators, as well as the difference equations in $z$ and $Q$,  
are all found from the their consistency with 
$q$-difference equations in variables $a_i$ that appear
in higher rank counts exactly as they appeared in \eqref{GTA}. 
The difference equations in $a_i$ are 
identified as the quantum Knizhnik-Zamolodchikov 
equations \cite{FrenResh} for $\cU_\hbar$ in  \cite{PCMI}, and this is
really the point through which the powerful machinery of geometric
representations theory enters the world of counting.

%\bibspread

\end{document}